\documentclass[11pt]{dgpaper}

\begin{document}

\title{\sc Symplectic excision}
\author{Yael Karshon \qquad Xiudi Tang}
\date{\today}
\maketitle

\begin{abstract}
  We use time-independent incomplete Hamiltonian flows to excise interesting closed subsets of positive codimension from symplectic manifolds.
  Examples of such subsets include what we call a ``Cantor brush'', a ``box with a tail'', and -- more generally -- epigraphs of lower semicontinuous functions.
  This answers a question of Alan Weinstein about excision of a ray, and it generalizes a result of Bernd Stratmann about excision of the product of a ray with a manifold.
\end{abstract}


\section{Introduction}
\label{sec:intro}

A few years ago, Alan Weinstein circulated the following question, which appeared in his paper with Christian Blohmann \cite[Question~11.2]{MR4727523}:
\begin{quotation}
  Let $(M, \omega)$ be a noncompact symplectic manifold, and let $[0, \infty) \cong R \subset M$ be a properly embedded ray.
  Is $M \setminus R$ symplectomorphic to $M$?
  Can the symplectomorphism be chosen to be the identity outside a prescribed neighbourhood $U$ of $R$?
\end{quotation}

Leonid Polterovich\footnote{In private communication} noted that a symplectic excision of a ray $[0, \infty) \times \Set{0}^{2n-1}$ from $\R^{2n}$ arises from the symplectization of a contactomorphism from the punctured standard unit sphere to the standard Euclidean space as a contact version of the stereographic projection constructed in \cite[Proposition~2.13]{MR2194671}.
However, this symplectomorphism is non-trivial everywhere outside the opposite ray; symplectization cannot yield a symplectomorphism that is supported in an arbitrarily small neighbourhood of the ray.

Xiudi Tang \cite{MR4216706} gave, for each $\eps > 0$, an explicit symplectic excision of the ray $[0, \infty) \times \Set{0}^{2n-1}$ from $\R^{2n}$ that is the identity map outside a neighbourhood of the ray that depends on $\eps$.
These neighbourhoods get smaller when $\eps$ gets smaller, and their intersection is the ray itself, but their volumes are infinite.
This result cannot be used to excise a ray from a non-compact symplectic manifold whose volume is finite.

During our preparation of this paper, Stratmann released his paper \cite{MR4537129}.
Stratmann excises what he calls ``parametrized rays'': closed submanifolds-with-boundary of the form $[0, \infty) \times \Sigma$ on which the two-form is the pullback of a two-form on $\Sigma$.
Stratmann achieves this result through a limit of a sequence of time-dependent Hamiltonian flows that coincide on larger and larger sets.
We achieve the same result through a single time-independent Hamiltonian flow.
Moreover, we use time-independent flows to excise more general closed subsets of non-compact symplectic manifolds.%
\footnote{\label{fn:topology}If $M$ is a manifold, $Z \subset M$ is a non-empty subset, and $M \setminus Z$ is homeomorphic to $M$, then $Z$ is closed in $M$ and $M$ is non-compact.}
For example, we excise the ``Cantor brush'' $\Set{0}^{2n-3} \times C \times [0, \infty) \times \Set{0}$, where $n \geq 2$ and $C \subset [0, 1]$ is the Cantor set (see \cref{exm:Cantor-brush}), or the ``box with a tail'' $[-1, 1]^{2n-2} \times [0, \infty) \times \Set{0} \cup \Set{0}^{2n-2} \times [-1, \infty) \times \Set{0}$ (see \cref{exm:box-w-tail}).
Using iterations of such flows, we excise even more general subsets, for example, a ``ray with two horns'' (see \cref{cor:ray-two-horns} and \cref{rem:ray-two-horns-auto}).

Our method is a novel variant of the symplectic isotopy extension theorem.
We assume that there is a submanifold $(N, \omega_N)$ of $(M, \omega)$ that contains the closed subset $Z$ and a null vector field on $(N, \omega_N)$ whose time-1 flow sends all the points of $Z$ to infinity and sends $N \setminus Z$ onto $N$, and we extend it to a Hamiltonian flow on $M$ whose time-$1$ flow sends $M \setminus Z$ onto $M$.
The subtle point is to ensure that the points of $M \setminus Z$ do not go to infinity in time $\leq 1$.

Following some preliminaries in \cref{sec:prelim}, we start with the simplest non-trivial case, of excising a ray from a symplectic manifold, in \cref{sec:ray};
see \cref{cor:remove-ray}.
This implies that any exact symplectic form has a nowhere vanishing primitive, as noted by Blohmann and Weinstein and proved differently by Stratmann; see \cref{rem:Blohmann-Weinstein} and \cref{thm:primitive}.
In \cref{sec:Hamiltonian-is-local} we prove locality of Hamiltonian excision: for a closed subset $Z$ of a symplectic manifold $M$, whether $Z$ is excisable by a time-$1$ Hamiltonian flow of a function that vanishes on $Z$ depends only on a neighbourhood of $Z$ in $M$; see \cref{thm:locality-excision}.
In \cref{sec:flows} we develop more subtle properties of flows on manifolds, including an interesting variant of the escape lemma (\cref{lem:escape-3}), leading to a technical lemma (\cref{lem:slow-zone}) that plays a crucial role in the following section.
In \cref{sec:smooth-to-hamiltonian} we show how to extend an excision by a null vector field on a submanifold into a Hamiltonian excision on the ambient symplectic manifold; see \cref{thm:extension-vector-field-germ}.
Combining this with the locality of Hamiltonian excision, we recover Stratmann's theorem; see \cref{thm:stratmann}.
In \cref{sec:smooth-epigraph-excision}, we construct smooth excisions by time-$1$ maps of vector fields.
Specifically, we consider subsets of $\Sigma \times (0, 1)$ of the form $Z := \Set{(p, x) \mmid x \geq \lambda(p)}$ for a function $\lambda \colon \Sigma \to (0, 1]$, where $\Sigma$ is any manifold.
It is not difficult to show that a necessary condition for excising $Z$ from $\Sigma \times (0, 1)$ is that the function $\lambda$ be lower semi-continuous; see \cref{lem:Z-closed}.
Our main result of this section is that lower semi-continuity of $\lambda$ is also sufficient; see \cref{thm:epigraph-smooth-excision}.
We first give a special case of this result that is easier to prove and that suffices for many applications; see \cref{lem:epigraph-smooth-forward-time}.
We then prove \cref{thm:epigraph-smooth-excision} through an iterative construction of the vector fields that excise the epigraphs of an increasing sequence of smooth functions that approximates the lower semi-continuous function.
Finally, in \cref{sec:example} we give some of the examples that we mentioned earlier.

Throughout this paper, ``manifold'' means ``smooth manifold'' and ``embedding'' means ``smooth embedding''.

\subsection*{Acknowledgement}

We are grateful to Alan Weinstein, whose question about excising a ray prompted this project. Xiudi Tang thanks Reyer Sjamaar for helpful discussions and thanks Ning Jiang for the warm hospitality at Wuhan University, where some ideas in this paper were sparked.
X.\ Tang thanks the Beijing Institute of Technology for a Research Fund Program for Young Scholars and National Natural Science Foundation of China for Young Scientist Fund 1220011806.
This research was partly funded by Natural Sciences and Engineering Research Council of Canada Discovery Grants RGPIN-2018-05771 and 485904.
Y.\ Karshon's research is partly funded by the United States -- Israel Binational Science Foundation.

\section{Preliminaries}
\label{sec:prelim}

\subsection*{Boundedness}
An \emph{exhaustion function} on a manifold $M$ is a continuous function $\zeta \colon M \to \R$ that is proper and bounded from below.
Such a function always exists (see \cite{MR2680546}).
A subset $B$ of $M$ is precompact (namely, its closure in $M$ is compact) if and only if its image $\zeta(B)$ is bounded in $\R$ for one, hence any, exhaustion function $\zeta$.

\subsection*{Semi-continuity}

$[-\infty, \infty]$ is the ordered set that is obtained from $\R$ by adding a maximal element $\infty$ and a minimal element $-\infty$.
For $u, v \in [-\infty, \infty]$ with $u \leq v$, we denote $[u, v] := \Set{x \mmid u \leq x \leq v}$.
A subset $J$ of $[-\infty, \infty]$ is \emph{convex} if for any $u, v \in [-\infty, \infty]$ with $u \leq v$, if $u, v \in J$, then $[u, v] \subseteq J$.
A map $f \colon N \to J$ from a topological space $N$ to a convex subset $J$ of $[-\infty, \infty]$ is \emph{lower semi-continuous} if for each $y \in J$ the set $\Set{p \in N \mmid f(p) > y}$ is open in $N$, and it is \emph{upper semi-continuous} if for each $y \in J$ the set $\Set{p \in N \mmid f(p) < y}$ is open in $N$.

\begin{lemma} \label{lem:semicontinuous-to-interval}
  Let $f \colon N \to J$ be a map from a topological space $N$ to a convex subset $J$ of $[-\infty, \infty]$.
  Then $f$ is lower semi-continuous as a map to $J$ if and only if it is lower semi-continuous as a map to $[-\infty, \infty]$.
  \ynote{\quad We use \cref{lem:semicontinuous-to-interval} in the proof of \cref{cor:D-open} with $(0, \infty]$ and $(-\infty, 0]$, and in~\cref{fn:to-J} with $(0, 1]$.}
\end{lemma}

\begin{proof}
  If the set $\Set{p \in N \mmid f(p) > y}$ is open for all $y \in [-\infty, \infty]$, then in particular it's open for all $y \in J$.
  For the other direction note that, because $J$ is convex, for $y$ outside~$J$, the set $\Set{p \in N \mmid f(p) > y}$ is either empty or all of $N$, so it's automatically open.
\end{proof}

Recall that the \emph{order topology} on an ordered set is the topology generated by the sub-basis consisting of the sets $\Set{x \mmid x > a}$ and $\Set{x \mmid x < b}$ for $a, b$ in the ordered set.
The order topology on $\R$ coincides with its standard topology, as well as with its relative topology that is induced from the order topology on $[-\infty, \infty]$.

\begin{lemma} \label{lem:semi-cont}
  Let $f \colon N \to J$ be a map from a topological space $N$ to a convex subset $J$ of $[-\infty, \infty]$.
  Then $f$ is lower semi-continuous if and only if the set
  \begin{equation} \label{eq:lower-level}
    \Set{(p, x) \in N \times J \mmid x < f(p)}
  \end{equation}
  is open in $N \times J$, and $f$ is upper semi-continuous if and only if the set
  \begin{equation*}
    \Set{(p, x) \in N \times J \mmid x > f(p)}
  \end{equation*}
  is open in $N \times J$.
  \ynote{\\ We use \cref{lem:semi-cont} in \cref{cor:D-open} with $J = [-\infty, \infty]$ and in \cref{lem:Z-closed} with $J = (0, 1]$.}
\end{lemma}

\begin{proof}
  We prove the first claim; the second claim is similar.
  Suppose that the set~\cref{eq:lower-level} is open in $N \times J$.
  Then for each $y \in J$, the preimage of the set~\cref{eq:lower-level} under the continuous map $p \mapsto (p, y)$ is open in $N$.
  This preimage is the set $\Set{p \in N \mmid f(p) > y)}$.
  Because $y$ is arbitrary, $f$ is lower semi-continuous.
  Conversely, suppose that $f$ is lower semi-continuous.
  Let $(p, x)$ be in the set~\cref{eq:lower-level}.
  Let $y \in J$ be such that $x < y < f(p)$.
  Then $p$ is in the subset $U := \Set{p' \in N \mmid y < f(p')}$ of $N$, which is open because $f$ is lower semi-continuous, and $x$ is in the subset $V := \Set{x' \in J \mmid x' < y}$ of $J$, which is open by the definition of the order topology.
  The product $U \times V$ is then an open subset of $N \times J$ that contains $(p, x)$ and is contained in the set~\cref{eq:lower-level}.
  Because $(p, x)$ is arbitrary, the set~\cref{eq:lower-level} is open.
\end{proof}

\begin{corollary} \label{cor:D-open}
  Let $N$ be a topological space, and let $D \subset N \times \R$ be a subset of the form
  \begin{equation*}
    D = \Set{(p, x) \mmid S(p) < x < T(p)}
  \end{equation*}
  for functions
  \begin{equation*}
    S \colon N \to [-\infty, 0) \quad \stext{and} \quad T \colon N \to (0, \infty].
  \end{equation*}
  Then $D$ is open in $N \times \R$ if and only if $S$ is upper semi-continuous and $T$ is lower semi-continuous.
  \ynote{\quad We refer to \cref{cor:D-open} in text in the subsection ``Flows''.}
\end{corollary}

\begin{proof}
  Because $\R$ is open in $[-\infty, \infty]$, $D$ is open in $N \times \R$ if and only if it is open in $N \times [-\infty, \infty]$.
  This holds if and only if its complement in $N \times [-\infty, \infty]$ is closed in $N \times [-\infty, \infty]$.
  This, in turn, holds if and only if the intersections of this complement with the closed subsets $N \times [-\infty, 0]$ and $N \times [0, \infty]$ are closed in $N \times [-\infty, \infty]$.
  These intersections are
  \begin{equation*}
    \Set{(p, x) \mmid x \leq S(p)} \quad \stext{and} \quad \Set{(p, x) \mmid x \geq T(p)}.
  \end{equation*}
  By \cref{lem:semi-cont}, the first of these sets is closed if and only if the function $S$ is upper semi-continuous to $[-\infty, \infty]$ and the second of these sets is closed if and only if the function $T$ is lower semi-continuous to $[-\infty, \infty]$.
  By \cref{lem:semicontinuous-to-interval}, these conditions hold iff $S$ is upper semi-continuous to $[-\infty, 0)$ and $T$ is lower semi-continuous to $(0, \infty]$.
\end{proof}

\subsection*{Flows}

An \emph{interval} is a convex subset of $\R$.
It is \emph{nondegenerate} if it is not empty nor a singleton.
Fix a vector field $Y$ on a manifold $N$.
A \emph{trajectory} of $Y$ is a differentiable map $\gamma \colon I \to N$ from an interval $I$ to $N$ such that if $I$ is non-degenerate then $\frac{\der\gamma}{\der t} = \Res{Y}_{\gamma(t)}$ for all $t \in I$.
The \emph{flow domain} of $Y$ is the set
\begin{multline*}
  D_Y := \{(x, t) \in N \times \R \mid \! \stext{there exists a trajectory $\gamma \colon I \to N$ of $Y$} \\
  \stext{whose domain $I$ contains $0$ and $t$ and such that $\gamma(0)=x$}\!\}.
\end{multline*}
This set has the form
\begin{equation*}
  D_Y = \Set{(x, t) \in N \times \R \mmid S_Y(x) < t < T_Y(x)}
\end{equation*}
for functions
\begin{equation*}
  T_Y \colon N \to (0, \infty] \quad \stext{and} \quad S_Y \colon N \to [-\infty, 0)
\end{equation*}
that are, respectively, lower semi-continuous and upper semi-continuous.
We call these functions, respectively, the \emph{forward exit time} and the \emph{backward exit time}.
Their semi-continuity properties are equivalent to $D_Y$ being open in $N \times \R$; see \cref{cor:D-open}.
There exists a (necessarily unique) smooth map
\begin{equation*}
  \Phi_Y \colon D_Y \to N,
\end{equation*}
called the \emph{maximal flow} of $Y$, such that the following holds.
For each $x \in N$, the curve
\begin{equation} \label{eq:max-traj}
  \gamma := \Phi_Y(x, \cdot) \colon (S_Y(x), T_Y(x)) \to N
\end{equation}
is a trajectory of $Y$ with initial condition $\gamma(0) = x$,
and for each trajectory
$\gamma \colon I \to N$ of $Y$ with $0 \in I$ and $\gamma(0) = x$, we have $I \subset (S_Y(x), T_Y(x))$ and $\gamma(\cdot) = \Phi_Y(x, \cdot)$.

For each $t \in \R$, we have the following diffeomorphism between open subsets of $N$, with inverse $\Phi_Y(\cdot, -t)$:
\begin{equation} \label{eq:Phi-t}
  \Phi_Y(\cdot, t) \colon \Set{x \in N \mmid S_Y(x) < t < T_Y(x)} \xrightarrow{\cong} \Set{x \in N \mmid S_Y(x) < -t < T_Y(x)}.
\end{equation}
For all $(z, t) \in D_Y$, we have
\begin{equation} \label{eq:exit-time-shift}
  T_Y(\Phi_Y(z, t)) = T_Y(z) - t \quad \stext{and} \quad S_Y(\Phi_Y(z, t)) = S_Y(z) - t.
\end{equation}

The map~\cref{eq:max-traj} is called the \emph{maximal trajectory of $Y$ that starts at $x$}.
The maps
\begin{equation*}
  \Res{\Phi_Y(x, \cdot)}_{[0, T_Y(x))} \colon [0, T_Y(x)) \to N \qquad \stext{and} \qquad \Res{\Phi_Y(x, \cdot)}_{(S_Y(x), 0]} \colon (S_Y(x), 0] \to N,
\end{equation*}
are called the \emph{maximal forward trajectory of $Y$ that starts at $x$} and the \emph{maximal backward trajectory of $Y$ that ends at $x$}.
For each $t \in \R$, the map~\cref{eq:Phi-t} is called the \emph{time-$t$ flow} of~$Y$.

We say that the time-$1$ flow of $Y$ \emph{excises} a subset $Z$ from $N$ if this time-$1$ flow is a diffeomorphism from $N \setminus Z$ to $N$.

\begin{lemma} \label{lem:time-one-excise}
  The time-$1$ flow of $Y$ excises $Z$ from $N$ if and only if
  \begin{equation*}
    Z = \Set{x \in N \mmid T_Y(x) \leq 1} \qquad \stext{and} \qquad S_Y(x) < -1 \text{ for all } x \in M.
  \end{equation*}
  If this holds, then the maximal forward trajectories that start in $Z$ stay in $Z$, and $S_Y(x) = -\infty$ for all $x \in M$.
  \ynote{ \\ We refer to \cref{lem:time-one-excise} in the proofs of \cref{remove-ray-standard,lem:cutoff-smooth,prop:ham-imply-nbhd,thm:locality-excision,thm:extension-vector-field-germ,thm:epigraph-smooth-excision,prop:exc-epigraph}.}
\end{lemma}

\begin{proof}
  The first claim follows from~\cref{eq:Phi-t}.
  The second claim then follows from the first part of~\cref{eq:exit-time-shift} with $t \geq 0$.
  For the third claim, rewrite the second part of~\cref{eq:exit-time-shift} with $t = -1$ as $S_Y(z) = S_Y(\Phi_Y(z, -1)) - 1$ to prove by induction that $S_Y < -n$ for all $n \in \N$.
\end{proof}

\begin{lemma}[Escape lemma, Version 1] \label{lem:escape-1}
  For any $x \in N$, if the maximal forward trajectory of $Y$ that starts at $x$ is contained in a compact subset of $N$, then it is defined for all positive times, and if the maximal backward trajectory of $Y$ that ends at $x$
  is contained in a compact subset of $N$, then it is defined for all negative times.
  \ynote{ \\ We refer to Lemma \ref{lem:escape-1} in the proofs of \cref{lem:flow-interval,lem:criterion-max-traj,lem:defined-all-times-nbhd,lem:escape-3}, and in text at the beginning of \cref{sec:flows}.}
\end{lemma}

\begin{proof}
  See \cite[Lemma~9.19]{MR2954043}.
\end{proof}

The \emph{Hamiltonian vector field} $X_F$ of a smooth function $F \colon M \to \R$ on a symplectic manifold $(M, \omega)$ is defined by $X_F \intprod \omega = \der F$; the \emph{Hamiltonian flow} of $F$ is the flow of~$X_F$.
This flow preserves $\omega$ and $F$.

\section{Excising a ray}
\label{sec:ray}

\subsection{Excising a ray from \texorpdfstring{$\R^{2n}$}{Rn}}
\label{ssec:ray-cartesian}

Consider $\R^{2n}$ with coordinates $(x_1, y_1, \ldots, x_n, y_n)$, with the standard symplectic form $\omegacan \coloneqq \der x_1 \wedge \der y_1 + \ldots + \der x_n \wedge \der y_n$, and, in it, consider the ray
\begin{equation*}
  R_0 \coloneqq \Set{0}^{2n-2} \times [0, \infty) \times \Set{0}.
\end{equation*}
\ynote{We refer to \cref{remove-ray-standard} in text above \cref{cor:remove-ray} and in the proof of \cref{cor:remove-ray}.}

\begin{proposition} \label{remove-ray-standard}
  For every neighbourhood of the ray $R_0$ in $\R^{2n}$, there exists a smooth function $\R^{2n} \to \R$ that is supported in the given neighbourhood, that vanishes on $R_0$, and whose time-$1$ Hamiltonian flow excises $R_0$ from $(\R^{2n}, \omegacan)$.
\end{proposition}

\begin{proof}
  Because there is a symplectomorphism $(-1, 1) \times \R \to \R^2$ that takes $[0, 1) \times \Set{0}$ to $[0, \infty) \times \Set{0}$, (for example, take the cotangent lift of the diffeomorphism $t \mapsto t/(1 - t^2)$ from $(-1, 1)$ to $\R$), it is enough to find, for
  \begin{equation*}
    M \coloneqq \R^{2n-2} \times (-1, 1) \times \R \quad \stext{and} \quad R_1 \coloneqq \Set{0}^{2n-2} \times [0, 1) \times \Set{0},
  \end{equation*}
  a smooth function $F \colon M \to \R$ that is supported in a given neighbourhood of $R_1$ in $M$, that vanishes on $R_1$, and whose time-$1$ Hamiltonian flow excises $R_1$ from $M$.%

  We write points of $M$ as
  \begin{equation*}
    z = (p; x_n, y_n) \quad \stext{with} \quad p = (x_1, y_1, \dotsc, x_{n-1}, y_{n-1}).
  \end{equation*}
  When $n = 1$ we use the same notation, with $p = 0$.

  Let $\eps \in (0, 1)$, and let
  \begin{equation*}
    h \colon [-\eps, 1) \to (0, \infty)
  \end{equation*}
  be a smooth strictly decreasing function that converges to $0$ at $1$.
  Let
  \begin{equation*}
    U \coloneqq \Set{(p; x_n, y_n) \mmid x_n \in (-\eps, 1) \stext{and} \abs{p}^2 + y_n^2 < h(x_n)}.
  \end{equation*}
  Then $U$ is an open neighbourhood of $R_1$ in $M$, and $\abs{p}^2 + y_n^2 < h(-\eps)$ on $U$.

  Fix a smooth function
  \begin{equation*}
    \chi \colon M \to [0, 1]
  \end{equation*}
  that is supported on the intersection of $U$ with the given neighbourhood of $R_1$ and is equal to $1$ on some smaller neighbourhood of $R_1$.
  Assume that $\der \chi = 0$ wherever $\chi = 0$; this can be achieved, for instance, by replacing $\chi$ by $\chi^2$.

  Let
  \begin{align*}
    F(z) \coloneqq \frac{1 - x_n^2}{\abs{p}^2 + 1 - x_n^2} \chi(z) y_n.
  \end{align*}
  Fix $c > 0$.
  Because the function $F$ is the product of $y_n$ with a function that takes values in $[0, 1]$ and is supported in $U$, if $\abs{F(z)} \geq c > 0$, then $z \in U$ and ${\abs{y_n} \geq c}$.
  From the definition of $U$, this further implies that $x_n \in [-\eps, 1)$ and $h(x_n) \geq c^2$.
  Because $h(x) \to 0$ as $x \to 1$, these inequalities imply that $x_n \in [-\eps, b]$ for some $b \in [-\eps, 1)$.
  Hence, the set $\Set{\abs{F} \geq c}$ is compact, as it is closed in~$M$ and contained in $\Set{x_n \in [-\eps, b], \ \abs{p}^2 + y_n^2 \leq h(-\eps)}$.
  Varying $c$, we conclude that
  \begin{equation*}
    \Res{F}_{M \setminus F^{-1}(0)} \colon M \setminus F^{-1}(0) \to \R \setminus \Set{0}
  \end{equation*}
  is a proper map.

  Let $X_F$ be the Hamiltonian vector field of $F$, and let $S_{X_F}$ and $T_{X_F}$ be, respectively, the backward exit time and the forward exit time of the flow of $X_F$.
  We claim that
  \begin{itemize}
    \item $T_{X_F} > 1$ on $M \setminus R_1$;
    \item $T_{X_F} \leq 1$ on $R_1$;
    \item $S_{X_F} < -1$ everywhere on $M$.
  \end{itemize}

  \noindent\emph{On $\Set{y_n \neq 0}$}:
  Let $z = (p; x_n, y_n)$ with $y_n \neq 0$.
  If $F(z) = 0$, then $\chi(z) = 0$; by the choice of $\chi$, also $\Res{\der \chi}_z = 0$; so $X_F(z) = 0$, and so the maximal trajectory of $X_F$ starting at $z$ is constant.
  If $F(z) \neq 0$, then the maximal trajectory of $X_F$ starting at $z$ contained in a level set of the proper map $\Res{F}_{M \setminus F^{-1}(0)} \colon M \setminus F^{-1}(0) \to \R \setminus \Set{0}$.
  In either case, the maximal trajectory of $X_F$ starting at $z$ is defined for all times, so $S_{X_F}(z) = -\infty$ and $T_{X_F}(z) = \infty$.

  \noindent\emph{On $\Set{y_n = 0}$:}
  At each point $z=(p; x_n, 0)$, we have
  \begin{align*}
    X_F(p; x_n, 0) = \frac{1 - x_n^2}{\abs{p}^2 + 1 - x_n^2} \chi(p; x_n, 0) \del_{x_n}.
  \end{align*}
  Since $X_F(p; x_n, 0)$ is a non-negative multiple of $\del_{x_n}$ and vanishes for $x_n < -\eps$, we have $S_{X_F}(p; x_n, 0) = -\infty$.
  It remains to calculate $T_{X_F}(p; x_n, 0)$.

  \noindent\emph{On $\Set{y_n = 0} \cap \Set{p \neq 0}$:}
  At each point $z=(p; x_n, 0)$ with $p \neq 0$, we have $\abs{p}^2 > 0$.
  By the comparison
  \begin{equation*}
    \frac{1 - x_n^2}{\abs{p}^2 + 1 - x_n^2} \chi(p; x_n, 0)
    \leq \frac{1 - x_n^2}{\abs{p}^2 + 1 - x_n^2}
  \end{equation*}
  and the completeness of the vector field $\frac{1 - x^2}{b + 1 - x^2} \del_{x}$ on $(-1, 1)$ for $b > 0$, we have $T_{X_F}(p; x_n, 0) = \infty$.

  \noindent\emph{On $\Set{y_n = 0} \cap \Set{p = 0}$:}
  At each point $z=(0; x_n, 0)$, we have
  \begin{equation*}
    X_F(z) = \chi(0; x_n, 0) \del_{x_n}.
  \end{equation*}
  When $0 \leq x_n < 1$, we have $X_F(0; x_n, 0) = \del_{x_n}$, so $T_{X_F}(0; x_n, 0) = 1-x_n$.
  Because $X_F(0; x_n, 0)$ is a positive multiple of~$\del_{x_n}$ near $x_n = 0$ and is a non-negative multiple of~$\del_{x_n}$ everywhere, the function $x_n \mapsto T_{X_F}(0; x_n, 0)$ is strictly decreasing near $x_n = 0$ and is weakly decreasing everywhere on $(-1, 1)$.
  We conclude that, for all such~$z$, we have $T_{X_F}(z) \leq 1$ if and only if $x_n \geq 0$.

  We have now shown that $T_{X_F}(z) \leq 1$ if and only if $z \in R_1$, and that $S_{X_F}(z) = -\infty$ for all $z \in M$.
  This proves our claim.
  The theorem then follows by \cref{lem:time-one-excise}.
\end{proof}

\subsection{Excising a ray from a symplectic manifold}

We will use the following topological lemma:

\begin{lemma} \label{lem:closed-nbhd}
  Let $g \colon M_1 \to M_2$ be a continuous map of manifolds, let $Z$ be a closed subset of $M_1$, and suppose that the restriction $g|_Z \colon Z \to M_2$ is proper.
  Then there exists a closed neighbourhood $C$ of $Z$ in $M_1$ such that the restriction $g|_{C} \colon C \to M_2$ is proper.
  \ynote{ \\ We use \cref{lem:closed-nbhd} in the proof of \cref{cor:remove-ray}.}
\end{lemma}

\begin{proof}
  Let $\zeta_1 \colon M_1 \to [0, \infty)$ and $\zeta_2 \colon M_2 \to [0, \infty)$ be exhaustion functions on $M_1$ and on $M_2$, respectively.
  Let
  \begin{multline*}
    C := \{y \in M_1 \mid \!\stext{there exists} x \in Z \stext{such that} \\
    \abs{\zeta_1(y) - \zeta_1(x)} \leq 1 \stext{and} \abs{\zeta_2(g(y)) - \zeta_2(g(x))} \leq 1\}.
  \end{multline*}
  For each $x \in Z$, the set
  \begin{equation*}
    \Set{y \in M_1 \mmid \abs{\zeta_1(y) - \zeta_1(x)} < 1 \stext{and} \abs{\zeta_2(g(y)) - \zeta_2(g(x))} < 1}
  \end{equation*}
  is an open subset of $M_1$ that contains $x$ and is contained in $C$.
  So $C$ is a neighbourhood of $Z$ in $N$.

  We will show that $C$ is closed in $M_1$.
  Let $(y_n)$ be a sequence in $C$ that converges in $M_1$ to a limit $y_\infty$.
  By the definition of~$C$, we can find a sequence $(x_n)$ in $Z$ such that, for each $n$,
  \begin{equation*}
    \abs{\zeta_1(y_n) - \zeta_1(x_n)} \leq 1 \quad \stext{and} \quad \abs{\zeta_2(g(y_n)) - \zeta_2(g(x_n))} \leq 1.
  \end{equation*}
  By continuity, $\zeta_1(y_n) \to \zeta_1(y_\infty)$.
  So $\Set{\zeta_1(y_n)}$ is bounded.
  This, in turn, implies that $\Set{\zeta_1(x_n)}$ is bounded.
  After passing to a subsequence, we may assume that $(x_n)$ converges to some $x_\infty \in M_1$.
  Because $Z$ is closed in $M$, we have $x_\infty \in Z$.
  By continuity,
  \begin{equation*}
    \abs{\zeta_1(y_\infty) - \zeta_1(x_\infty)} \leq 1 \quad \stext{and} \quad \abs{\zeta_2(g(y_\infty)) - \zeta_2(g(x_\infty))} \leq 1.
  \end{equation*}
  So $y_\infty$ is in $C$.

  To show that $\Res{g}_C \colon C \to M_2$ is proper, it now suffices to show that for every $b > 0$ there exists $B > 0$ such that for every $y \in C$, if $\zeta_2(g(y)) \leq b$, then $\zeta_1(y) \leq B$.
  So let $b > 0$.
  Because $g|_Z \colon Z \to M_2$ is proper, the set $K := \Set{x \in Z \mmid \zeta_2(g(z)) \leq 1 + b}$ is compact.
  We take $B := \sup \Set{1 + \zeta_1(x) \mmid x \in K}$.
  Now, let $y \in C$, and suppose that $\zeta_2(g(y)) \leq b$.
  By the definition of $C$, we can choose $x \in Z$ such that $\abs{\zeta_1(y) - \zeta_1(x)} \leq 1$ and $\abs{\zeta_2(g(y)) - \zeta_2(g(x))} \leq 1$.
  Then $\zeta_2(g(x)) \leq 1 + \zeta_2(g(y)) \leq 1 + b$, so $x \in K$, and $\zeta_1(y) \leq 1 + \zeta_1(x) \leq B$.
\end{proof}

The following corollary of \cref{remove-ray-standard} answers Weinstein and Blohmann's question about excising rays from arbitrary symplectic manifolds:

\begin{theorem} \label{cor:remove-ray}
  Let $(M, \omega)$ be a symplectic manifold, and let $R$ be a properly embedded ray in $M$.
  Then for every neighbourhood $U_R$ of $R$ in $M$ there exists a smooth function $M \to \R$ that is supported in $U_R$, that vanishes on $R$, and whose time-$1$ Hamiltonian flow excises $R$ from $M$.
  \ynote{ \\ We use \cref{cor:remove-ray} in \cref{rem:closed} and in the proofs of \cref{cor:ray-two-horns,thm:primitive}.}
\end{theorem}

\begin{proof}
  The Weinstein symplectic tubular neighbourhood theorem implies (see below) that there exist an open neighbourhood $U_0$ of the standard ray $R_0$ in $\R^{2n}$ and an open neighbourhood $U$ of $R$ in $M$ and a symplectomorphism $\psi \colon (U_0, \omegacan) \to (U, \omega)$ that takes $R_0$ to $R$.
  After replacing $U_0$ by its intersection with $\psi^{-1}(U_R)$, we may assume that $U \subset U_R$.
  By \cref{lem:closed-nbhd}, there exists a closed neighbourhood $C$ of $R_0$ in $U_0$ such that the restriction $\psi|_C \colon C \to M$ is proper.
  In particular, $\psi(C)$ is closed in $M$.
  By \cref{remove-ray-standard}, there exists a smooth function $F_0 \colon \R^{2n} \to \R$ whose support is contained in $C$, that vanishes on $R_0$, and whose time-$1$ Hamiltonian flow excises $R_0$ from $(\R^{2n}, \omegacan)$.
  The function $F \colon M \to \R$ that is defined by $F_0 \circ \psi^{-1}$ on $U$ and zero outside $U$ is then smooth (because it is smooth on the open sets $U$ and $M \setminus \psi(C)$, which cover $M$).
  The support of $F$ is contained in the given neighbourhood $U_R$ of $R$, and the time-$1$ Hamiltonian flow of $F$ is a symplectomorphism of $M \setminus R$ with~$M$.

  (The Weinstein symplectic tubular neighbourhood theorem is usually stated for submanifolds, not submanifolds-with-boundary.
  To apply it to the ray $R$ as above, we extend a parametrization $\gamma \colon [0, \infty) \to R$ of $R$ to an embedding $\gamma \colon (-\delta, \infty) \to M$ for some $\delta > 0$.
  We recall how to do this.
  By the definition of a smooth map on a manifold-with-boundary, $\gamma$ extends to a smooth map $\gamma \colon (-\eps, \infty) \to M$ for some $\eps > 0$.
  Because $\dot{\gamma}(0) \neq 0$, after shrinking $\eps$ we may assume that $\Res{\gamma}_{(-\eps, \eps)}$ is an embedding.
  Because $\Res{\gamma}_{[0, \infty)}$ is proper and one-to-one, $\gamma([\eps, \infty))$ is closed and disjoint from $\gamma(0)$, so there exists $0 < \delta < \eps$ such that $\gamma([-\delta, 0])$ is disjoint from $\gamma([\eps, \infty))$.
  The map $\Res{\gamma}_{[-\delta, \infty)} \to M$ is then a one-to-one proper immersion, hence an embedding.)
\end{proof}

\begin{remark} \label{rem:closed}
  In the proof of Corollary~\ref{cor:remove-ray}, it is important that the image under $\psi$ of the support of $F_0$ be closed in $M$ and not only in $U$.
  In the first arXiv version of this paper, as well as in Tang's earlier paper \cite[Corollary~1.2]{MR4216706}, we mistakenly omitted this detail of the proof.
\eor
\end{remark}

\begin{corollary} \label{cor:ray-two-horns}
  Let $(M, \omega)$ be a symplectic manifold.
  Let $Z_0$ be a ``ray with two horns'' in $\R^2$, obtained as the union of the non-negative $x$-axis $\Set{(x, 0) \mmid x \geq 0}$ and the closed intervals $\Set{(-s, s) \mmid s \in [0, 1]}$ and $\Set{(-s, -s) \mmid s \in [0, 1]}$.
  Let $Z$ be the image of a proper embedding of $Z_0$ into $M$.
  Then there exists a symplectomorphism from $M \setminus Z$ onto $M$.
  \ynote{ \\ We use \cref{cor:ray-two-horns} in \cref{rem:ray-two-horns-auto}.}
\end{corollary}

\begin{proof}
  Let $Z_1$, $Z_2$, $Z_3$, respectively, be the images in $M$ of the non-negative $x$-axis, of the interval $\Set{(-s, s) \mmid s \in [0, 1]}$, and of the interval $\Set{(-s, -s) \mmid s \in [0, 1]}$.
  Applying Corollary \ref{cor:remove-ray} three times, we obtain symplectomorphisms
  \begin{equation*}
    M \setminus Z = (M \setminus (Z_1 \cup Z_2)) \setminus Z_3 \ \to \ M \setminus (Z_1 \cup Z_2) = (M \setminus Z_1) \setminus Z_2 \ \to \ M \setminus Z_1 \ \to \ M.
  \end{equation*}
\end{proof}

\begin{figure}[htb]
  \centering
  \includegraphics[width=\textwidth, height = 2cm]{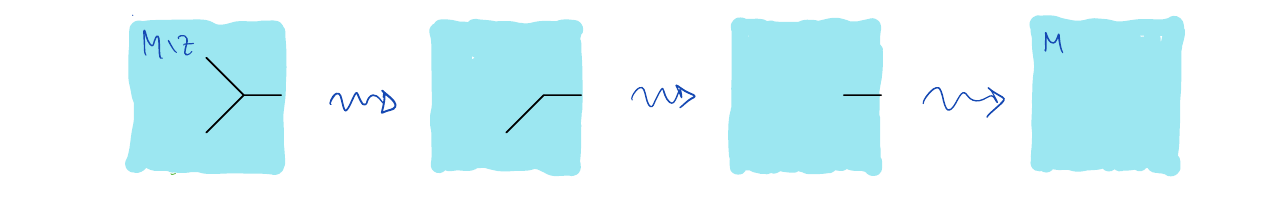}
  \caption{Removing a ray with two horns}
\end{figure}

\subsection{Nonvanishing primitive of exact symplectic forms}
\label{ssec:applications}

The original motivation for Weinstein and Blohmann's question was to show that every exact symplectic form has a nowhere vanishing primitive.
We show this in \cref{thm:primitive} below, after the preliminary \cref{lem:isolated-zero}.

\begin{lemma} \label{lem:isolated-zero}
  Any exact $2$-form has a primitive whose zero set is discrete.
  \ynote{ \\ We refer to \cref{lem:isolated-zero} in \cref{rem:Blohmann-Weinstein} and in the proof of \cref{thm:primitive}.}
\end{lemma}

\begin{proof}
  Let $\omega = \der \theta$ be an exact $2$-form on a manifold $M$.
  Let $\phi = (x_1, \dotsc, x_n) \colon M \to \R^n$ be an embedding.
  The smooth function
  \begin{equation*} \begin{split}
    F \colon M \times \R^n &\to T^*M, \\
    F(x, s) &= \theta(x) + \sum_{i = 1}^n s_i \der_x x_i,
  \end{split} \end{equation*}
  where $s = (s_1, \dotsc, s_n)$, is transverse to the zero section $0_{T^*M}$ of $T^*M$.
  By the Transversality Theorem (see, for instance, \cite{MR2680546}), we deduce that
  $F(\cdot, s) \colon M \to T^*M$ is transverse to $0_{T^*M}$ for almost every $s \in \R^n$.
  Fix such an $s$, and let $\rho := \sum_{i = 1}^n s_i x_i \colon M \to \R$.
  Then the zero set of the one-form  $\beta \coloneqq \theta + \der \rho = F(\cdot, s) \in \Omega^1(M)$ is a discrete set of points in~$M$.
  Because $\der \beta = \der \theta = \omega$, the one-form $\beta$ is a primitive of $\omega$.
\end{proof}

\begin{remark} \label{rem:Blohmann-Weinstein}
  Blohmann and Weinstein's paper \cite{MR4727523} contains a statement of \cref{lem:isolated-zero} without explanation and an idea for the proof of \cref{thm:primitive}.
  Our proof follows their ideas.
  Stratmann's paper \cite{MR4577600} contains a proof of \cref{thm:primitive} that does not rely on the excision of a ray.
\eor
\end{remark}

\begin{theorem} \label{thm:primitive}
  Any exact symplectic form has a nowhere vanishing primitive.
\end{theorem}

\begin{proof}
  Let $(M, \omega)$ be an exact symplectic manifold.
  By \cref{lem:isolated-zero}, $\omega$ has a primitive $\beta$ whose zeroes are isolated.
  Choose such a $\beta$, and fix an enumeration $(z_i)_{i \geq 1}$ of its zeroes.

  We construct an exhaustion of $M$ by a sequence of compact subsets $(K_j)_{j = 1}^\infty$ such that each connected component of each $M \setminus K_j$ has non-compact closure in $M$.
  This can be achieved by taking the unions of regular sublevel sets of an exhaustion function for $M$ with the precompact connected components of their complements.
  It implies that for each $j$ any point in $K_{j+1} \setminus K_j$ can be joined by a smooth path in $M \setminus K_j$ to a point in $M \setminus K_{j+1}$.
  (A similar argument appeared in \cite{MR542888,MR4039813,MR3901809,MR4537129}.)

  We now construct for each~$i$ a properly embedded ray $R_i$ starting at $z_i$ such that the rays are pairwise disjoint and each point of $M$ has a neighbourhood that meets only finitely many of the rays.
  Our construction is recursive in $j$.
  For each $j$, and for each point in $K_{j+1} \setminus K_j$ that is either in $\{z_i\}$ or an endpoint of a previous path, we draw a smooth path in $M \setminus K_j$ that connects that point to a point in $M \setminus K_{j+1}$, such that all the new paths are disjoint from each other, from previous paths, and from all the points $z_i$.
  Moreover, we arrange the paths to have non-zero velocity, and whenever we extend an earlier path, we arrange that the concatenated path will be smooth.

  Let $(U_i)_{i \geq 1}$ be pairwise disjoint open neighbourhoods of $(R_i)_{i \geq 1}$, such that each point of $M$ has a neighbourhood that meets only finitely many of the sets $U_i$.
  By \cref{cor:remove-ray}, for each $i$, there is a smooth function $F_i \colon M \to \R$ supported in $U_i$ whose time-$1$ Hamiltonian flow, $\varphi_i$, excises $R_i$ from $M$.
  Let $R := \bigcup_{i \geq 1} R_i$, and let $U := \bigcup_{i \geq 1} U_i$.
  Then $F := \sum_{i \geq 1} F_i$ is well defined and smooth, it is supported in $U$, and its time-$1$ flow $\varphi$ is the composition of $(\varphi_i)_{i \geq 1}$.
  Let $\alpha := (\varphi^{-1})^* (\Res{\beta}_{M \setminus R})$.
  Then
  \begin{equation*}
    \der \alpha = \der ((\varphi^{-1})^* (\Res{\beta}_{M \setminus R})) = (\varphi^{-1})^* (\Res{\omega}_{M \setminus R}) = \omega,
  \end{equation*}
  so $\alpha$ is a primitive of $\omega$.
  Moreover, $\alpha$ has no zeroes, because $\beta$ has no zeroes in $M \setminus R$.
\end{proof}

\section{Hamiltonian excision is local}
\label{sec:Hamiltonian-is-local}

In this section we show that if a subset $Z$ of a symplectic manifold is excisable by a time-$1$ Hamiltonian flow of a function that vanishes on $Z$, then the function can be chosen to be supported on an arbitrarily small neighbourhood of $Z$.
See \cref{prop:ham-imply-nbhd}.
We conclude that, for a closed subset $Z$ of a symplectic manifold $M$, whether $Z$ is excisable by a time-$1$ Hamiltonian flow of a function that vanishes on $Z$ depends only on a neighbourhood of $Z$ in $M$.
See \cref{thm:locality-excision}.

The challenge is to multiply the Hamiltonian function by a cut-off function in such a way that the resulting time-$1$ Hamiltonian flow does not send to infinity any points outside of $Z$.
The analogous result for not-necessarily-Hamiltonian smooth flows is easier; see \cref{cor:cutoff-smooth}; it relies on the escape lemma for flows of vector fields.
We begin with some preparatory lemmas, starting with some qualitative properties of flows on intervals.

\begin{lemma} \label{lem:flow-interval}
  Let $I$ be a (possibly unbounded) open interval containing the origin $0$, let $v \colon I \to [0, 1]$ be a smooth function, and let $\tau \colon J \to I$ be the maximal trajectory of the vector field $v(t) \del / \del t$ on $I$, starting at the origin.
  \begin{itemize}[itemsep = 6pt, topsep = 6pt]
    \item If $\tau$ is not defined for all positive times, then it restricts to a diffeomorphism from $J^+ := J \cap [0, \infty)$ to $I^+ := I \cap [0, \infty)$.
      Similarly, if $\tau$ is not defined for all negative times, then it restricts to a diffeomorphism from $J^- := J \cap (-\infty, 0]$ to $I^- := I \cap (-\infty, 0]$.
    \item The domain $J$ of $\tau$ contains the interval $I$.
  \end{itemize}
  \ynote{We use both parts of \cref{lem:flow-interval} in the proof of \cref{lem:reparametrize-0-1}.}
\end{lemma}

\begin{proof}
  If $v(\tau(s_0)) = 0$ for some $s_0 \in J$, then $\tau$ must coincide with the constant curve with value $\tau(s_0)$ and domain $\R$, and we are done.
  Otherwise, $\frac{\der\tau}{\der s} = v(\tau(s)) > 0$ for all $s \in J$, so $\tau$ is strictly monotone increasing.
  By the intermediate value theorem and the inverse function theorem, the image $\tau(J)$ is an open subinterval of $I$ and $\tau \colon J \to \tau(J)$ is a diffeomorphism.
  Because $\tau(0) = 0$, $\tau(J^+)=\tau(J) \cap I^+$ and $\tau(J^-)=\tau(J) \cap I^-$.

  If $J^+$ is not bounded, then $J^+ = [0, \infty)$ automatically contains $I^+$.
  Suppose now that $J^+$ is bounded.
  By the escape lemma (\cref{lem:escape-1}), $\tau(J_+)$ does not have an upper bound in $I$.
  By the previous paragraph, $\tau|_{J^+} \colon J^+ \to I^+$ is then a diffeomorphism.
  Because $\der \tau/\der s \in [0, 1]$, we have $0 \leq \tau(s) \leq s$ for all $s \in J^+$, so $\tau(J^+) \subset J^+$.
  But $\tau(J^+)=I^+$, so $I^+ \subset J^+$.
  The arguments for the backward flow are similar.
\end{proof}

\begin{lemma} \label{lem:criterion-max-traj}
  Let $Y$ be a vector field on a manifold $N$, let $z_0 \in N$, and let $\sigma^+ \colon J^+ \to N$ be a forward trajectory of $Y$ starting at $z_0$.
  Then $\sigma^+$ is the maximal forward trajectory of $Y$ starting at $z_0$ if and only if either $J^+ = [0, \infty)$ or the image of $\sigma^+$ is not contained in any compact subset of $N$.
  \ynote{ \\ We use \cref{lem:criterion-max-traj} in the proofs of \cref{cor:invariant-submfd,cor:max-traj-iff,lem:reparametrize-0-1}.}
\end{lemma}

\begin{proof}
  Suppose that $\sigma^+$ is not maximal.
  Then $\sigma^+$ extends to a forward trajectory $\delta^+$ of~$Y$ whose domain is strictly larger than $J^+$.
  In particular, $J^+$ cannot be all of $[0, \infty)$, so it has the form $J^+=[0, b)$ for some $b > 0$ in $\R$, and the domain of $\delta^+$ contains $[0, b]$.
  Since $\sigma^+$ coincides with $\delta^+$ on $[0, b)$, its image is contained in the compact subset $\delta^+([0, b])$ of $N$.
  The converse direction follows from the escape lemma (\cref{lem:escape-1}).
\end{proof}

We have the following two corollaries of \cref{lem:criterion-max-traj}.

\begin{corollary} \label{cor:invariant-submfd}
  Let $Y$ be a vector field on a manifold $M$.
  Denote its support by $S$.
  Let $N$ be a submanifold such that $Y$ is tangent to $N$ at each point of $N$.
  Suppose that $N \cap S$ is closed in $M$.
  Then $N$ is invariant under the flow of $Y$.
  \ynote{ \\ We use \cref{cor:invariant-submfd} in the proofs of \cref{lem:extension-vector-field-germ,thm:extension-vector-field-germ}.}
\end{corollary}

  \begin{proof}
  We would like to show that, for each $z \in N$, the maximal trajectory of $Y|_N$ that starts at $z$ is maximal also as a trajectory of $Y$.

  On the set $M \setminus S$, we have $Y=0$.
  So $M \setminus S$ is fixed under the flow of $Y$, and $S$ is invariant under the flow of $Y$.

  Fix $z \in N \cap S$.
  If the maximal forward trajectory of $Y|_N$ in $N$ that starts at $z$ is defined for all positive times, then as in \cref{lem:criterion-max-traj}, it is maximal also as a forward trajectory of $Y$ in $M$.
  Suppose now that it is not defined for all positive times.
  By \cref{lem:criterion-max-traj} for $Y|_N$, its image is not contained in any compact subset of $N$.
  But its image is contained in $N \cap S$, which is closed in $M$.
  So its image is not contained in any compact subset of $M$.
  By \cref{lem:criterion-max-traj} for $Y$, it is maximal also as a forward trajectory of $Y$ in $M$.

  A similar argument holds for the backward trajectories.
\end{proof}

  \begin{corollary} \label{cor:max-traj-iff}
  Let $Y_1$ and $Y_2$ be vector fields on a manifold $N$.
  Then every maximal forward trajectory for~$Y_1$ along which $Y_1$ and $Y_2$ coincide is also a maximal forward trajectory for $Y_2$, and every maximal backward trajectory for~$Y_1$ along which $Y_1$ and $Y_2$ coincide is also a maximal backward trajectory for $Y_2$.
  \ynote{\\ We use \cref{cor:max-traj-iff} in the proofs of \cref{lem:cutoff-smooth} (forward), \cref{prop:ham-imply-nbhd} (forward and backward), and \cref{lem:extension-vector-field-germ}.
  }
\end{corollary}

\begin{proof}
  Let $\gamma^+ \colon I^+ \to N$ be a maximal forward trajectory for $Y_1$ along which $Y_1$ and $Y_2$ coincide.
  For all $t \in I^+$, we have$\frac{\der \gamma^+}{\der t} = \Res{Y_1}_{\gamma^+(t)} = \Res{Y_2}_{\gamma^+(t)}$, so $\gamma^+$ is also a forward trajectory for $Y_2$.
  By \cref{lem:criterion-max-traj}, $\gamma^+$ is maximal as a forward trajectory for $Y_1$ if and only if it is maximal as a forward trajectory for $Y_2$.
  The argument for backward trajectories is similar.
\end{proof}

\begin{lemma} \label{lem:reparametrize-0-1}
  Let $Y$ be a vector field on a manifold $N$, let $\chi \colon N \to [0, 1]$ be a smooth function, and let $z_0 \in N$.
  Then the following holds.
  \begin{itemize}[itemsep = 6pt]
    \item The domain of the maximal trajectory of $\chi Y$ starting at $z_0$ contains the domain of the maximal trajectory of $Y$ starting at $z_0$.
    \item If the maximal trajectory of $\chi Y$ starting at $z_0$ is not defined for all positive times, then neither is the maximal trajectory of $Y$ starting at $z_0$, and the maximal forward trajectories of $Y$ and of $\chi Y$ then have the same image in $N$.
      A similar result holds for the maximal backward trajectories.
  \end{itemize}
  \ynote{We use \cref{lem:reparametrize-0-1} in the proofs of \cref{lem:cutoff-smooth,prop:ham-imply-nbhd,lem:extension-vector-field-germ}.}
\end{lemma}

\begin{proof}
  Let $\gamma \colon I \to N$ be the maximal trajectory of $Y$ starting at $z_0$.
  Let $\tau \colon J \to I$ be the maximal trajectory of the vector field $\chi(\gamma(t)) \frac{\del}{\del t}$ on the interval~$I$, starting at $0$.
  A direct calculation shows that $\gamma \circ \tau \colon J \to N$ is a trajectory of the vector field $\chi Y$ starting at $z_0$.
  Let $\sigma \colon K \to N$ be the \emph{maximal} trajectory of $\chi Y$ starting at $z_0$.
  By maximality and by the second item of \cref{lem:flow-interval}, we have
  \begin{equation*}
    K \supset J \supset I \quad \stext{and} \quad \gamma \circ \tau = \Res{\sigma}_J.
  \end{equation*}
  In particular, this proves the first item.

  Let
  \begin{equation*}
    K^+ := K \cap [0, \infty), \quad J^+ := J \cap [0, \infty), \quad \stext{and} \quad I^+ := I \cap [0, \infty).
  \end{equation*}
  Suppose that $K^+$ is bounded.
  Then so are $J^+$ and $I^+$.
  By the first item of \cref{lem:flow-interval}, $\tau(J^+)=I^+$.
  By \cref{lem:criterion-max-traj}, $\gamma(I^+)$ is not contained in any compact subset of $N$.
  But $\gamma(I^+) = \gamma(\tau(J^+)) = \sigma(J^+)$, so $\sigma(J^+)$ is not contained in any compact subset of $N$.
  \cref{lem:criterion-max-traj} then implies that $J^+=K^+$.
  So $\gamma(I^+)=\sigma(K^+)$.
  A similar argument holds if $K^-$ is bounded.
  This proves the second item.
\end{proof}

\begin{lemma} \label{lem:cutoff-smooth}
  Let $Z$ be a non-empty closed subset of a manifold $N$, let $Y$ be a vector field on~$N$ whose time-$1$ flow excises $Z$ from $N$, and let $\chi \colon N \to [0, 1]$ be a smooth function that is equal to $1$ on $Z$.
  Then the time-$1$ flow of $\chi Y$ also excises $Z$ from $N$.
  \ynote{We use \cref{lem:cutoff-smooth} in \cref{cor:cutoff-smooth} and in the proof of \cref{thm:extension-vector-field-germ}.}
\end{lemma}

\begin{proof}
  Fix any $z_0 \in N$.
  Let $I^-$ and $J^-$ be the domains of the maximal backward trajectories of $Y$ and of $\chi Y$ starting at $z_0$.
  By \cref{lem:time-one-excise}, $I^-=(-\infty, 0]$.
  By the first part of \cref{lem:reparametrize-0-1}, $J^- \supset I^-$.
  So $J^-=(-\infty, 0]$.

  Next, let
  \begin{equation*}
    \gamma^+ \colon I^+ \to N \quad \stext{and} \quad \sigma^+ \colon J^+ \to N
  \end{equation*}
  be the maximal forward trajectories of $Y$ and of $\chi Y$ starting at $z_0$.

  Suppose that $J^+ \subset [0, 1)$.
  By the first part of \cref{lem:reparametrize-0-1}, $I^+ \subset J^+$.
  So $I^+ \subset [0, 1)$, and by \cref{lem:time-one-excise}, $z_0 \in Z$.

  In the other direction, suppose that $z_0 \in Z$.
  By \cref{lem:time-one-excise}, $I^+ \subset [0, 1)$, and the maximal forward trajectory $\gamma^+$ stays in $Z$.
  So the vector fields $Y$ and $\chi Y$ coincide along~$\gamma^+$.
  By \cref{cor:max-traj-iff}, $\gamma^+ = \sigma^+$; in particular, $I^+ = J^+$.
  So $J^+ \subset [0, 1)$.

  We have shown that $J^- = (-\infty, 0]$, and that $ J^+ \subset [0, 1)$ if and only if $z_0 \in Z$.
  Because $z_0 \in N$ was arbitrary, by \cref{lem:time-one-excise}, the time-$1$ flow of $\chi Y$ excises $Z$ from~$N$.
\end{proof}

\begin{corollary} \label{cor:cutoff-smooth}
  Let $Z$ be a non-empty closed subset of a manifold $N$.
  Suppose that there exists a vector field on~$N$
  whose time-$1$ flow excises $Z$ from $N$.
  Then for any neighbourhood $U$ of $Z$ in $N$ there exists a vector field on $N$ that is supported in $U$ and whose time-$1$ flow excises $Z$ from $N$.
  Indeed, multiply the vector field by a smooth function $\chi \colon N \to [0, 1]$ that is equal to $1$ on $Z$ and is supported in $U$, and apply \cref{lem:cutoff-smooth}.
  \ynote{\\ We use \cref{cor:cutoff-smooth} in the proof of \cref{cor:hamiltonian-excision-new}.}
\end{corollary}

A continuous real-valued function $f$ on a topological space \emph{vanishes at infinity} if for every $\eps > 0$ the superlevel set $\Set{\abs{f} \geq \eps}$ is compact.
For example, the zero function vanishes at infinity, and the reciprocal of any positive exhaustion function vanishes at infinity.

We will use the following result.

\begin{lemma} \label{lem:vanish-at-infinity-in-nbhd-v1}
  Let $M$ be a manifold, $Z$ a closed subset of $M$, $U$ a neighbourhood of $Z$ in $M$, and $H \colon U \to \R$ a continuous function that vanishes on $Z$.
  Then there exists a closed neighbourhood $C$ of $Z$ in $M$ that is contained in $U$ and such that $\Res{H}_C \colon C \to \R$ vanishes at infinity.
  \ynote{\\ We refer to \cref{lem:vanish-at-infinity-in-nbhd-v1} in the proofs of \cref{prop:ham-imply-nbhd,lem:extension-vector-field-germ}.}
\end{lemma}

\begin{proof}
  Fix an exhaustion function $\zeta \colon M \to [1, \infty)$.
  Let $C'$ be a closed neighbourhood of~$Z$ in $M$ that is contained in $U$.
  Then
  \begin{equation*}
    C := \Set{y \in C' \mmid \abs{H(y)} \leq 1/\zeta(y)}
  \end{equation*}
  is closed in $M$ (because $H|_{C'}$ and $\Res{\zeta}_{C'}$ are continuous on $C'$ and $C'$ is closed in $M$), and it is a neighbourhood of $Z$ in $M$
  (because it contains the intersection of the neighbourhood $C'$ of $Z$ with the neighbourhood $\Set{y \in U \mmid \abs{H(y)} < 1/\zeta(y)}$ of $Z$).
  We claim that $\Res{H}_C$ vanishes at infinity.
  Indeed, let $\eps > 0$.
  Because $\zeta \colon M \to [1, \infty)$ is proper, the set
  \begin{equation*}
    K := \Set{y \in M \mmid \zeta(y) \leq 1/\eps}
  \end{equation*}
  is compact.
  If $y \in C$ and $\abs{H(y)} \geq \eps$, then $\eps \leq \abs{H(y)} \leq 1/\zeta(y)$, so $y \in K$.
  So the set $C \cap \Set{\abs{H} \geq \eps}$, being a closed subset of the compact set $K$, is compact.
  So $\Res{H}_C$ vanishes at infinity.
\end{proof}

\begin{lemma} \label{lem:defined-all-times-nbhd}
  Let $(M, \omega)$ be a symplectic manifold, ${H \colon M \to \R}$ a smooth function, and~$C$ a closed subset of $M$ such that $\Res{H}_C$ vanishes at infinity.
  Let $\chi \colon M \to [0, 1]$ be a smooth function whose support is contained in $C$ and such that $d\chi$ vanishes whenever $\chi$ vanishes.
  Let $F := \chi H$.
  Then every maximal trajectory of $X_F$ that starts in the set $C \cap H^{-1}(0)$ stays in this set, and every maximal trajectory of $X_F$ that starts outside this set is defined for all times.
  \ynote{We use \cref{lem:defined-all-times-nbhd} in the proofs of \cref{prop:ham-imply-nbhd,lem:extension-vector-field-germ}.}
\end{lemma}

\begin{proof}
  It is enough to show that every maximal trajectory of $X_F$ that starts in the set $M \setminus (C \cap H^{-1}(0))$ stays in this set and is defined for all times.
  Let $z \in M \setminus (C \cap H^{-1}(0))$.
  \begin{itemize} [itemsep = 6pt]
  \item
    Suppose that $\chi(z) = 0$.
    Then also $d\chi|_z = 0$, and so $\Res{X_F}_z = 0$.
    So the maximal trajectory of $X_F$ that starts at $z$ is constant and is defined for all times.
  \item
    Suppose that $\chi(z) \neq 0$.
    Then $z \in C$, and $H(z) \neq 0$.
    So $a := F(z) \neq 0$.
    So the maximal trajectory of $X_F$ that starts at $z$ is contained in the non-zero level set $\Set{F = a}$, which is contained in the set $M \setminus (C \cap H^{-1}(0))$.
    The function $F$ vanishes at infinity.
    (Indeed, for each $\eps > 0$, the set $\Set{\abs{F} \geq \eps}$ is closed, and it is contained in $C \cap \Set{\abs{H} \geq \eps}$, which is compact because $\Res{H}_C$ vanishes at infinity.)
    So the non-zero level set $\Set{F = a}$ is compact.
    By the escape lemma (\cref{lem:escape-1}), the maximal trajectory of $X_F$ that starts at $z$ is defined for all times.
  \end{itemize}
\end{proof}

\begin{proposition}[Restriction of Hamiltonian excision] \label{prop:ham-imply-nbhd}
  Let $(M, \omega)$ be a symplectic manifold, $Z$ a closed subset of $M$, and $H \colon M \to \R$ a smooth function that vanishes on $Z$.
  Assume that the time-$1$ Hamiltonian flow of $H$ excises $Z$ from $M$.
  Then for each neighbourhood $U$ of $Z$ in~$M$ there exists a smooth function $\chi \colon M \to [0, 1]$ that is equal to~$1$ on $Z$, whose support is contained in $U$, and such that the time-$1$ Hamiltonian flow of $F := \chi H$ excises $Z$ from $M$.
  \ynote{ \\ We use \cref{prop:ham-imply-nbhd} in the proofs of \cref{thm:locality-excision,thm:stratmann}.}
\end{proposition}

\begin{proof}
  Let $C$ be a closed neighbourhood of $Z$ in $M$ that is contained in $U$ and such that $\Res{H}_C$ vanishes at infinity; such a neighbourhood exists by \cref{lem:vanish-at-infinity-in-nbhd-v1}.
  Let $\chi \colon M \to [0, 1]$ be a smooth function that is equal to $1$ near $Z$ and is supported in the interior of $C$.
  Assume, moreover, that $d\chi$ vanishes wherever $\chi$ vanishes; this can be achieved, e.g., by replacing $\chi$ by $\chi^2$.
  Let $F := \chi H$.
  Let $X_H$ be the Hamiltonian vector field of $H$, and let $X_F$ be the Hamiltonian vector field of $F$.
  Let $z \in M$.
  \begin{itemize}[itemsep = 6pt]
    \item
      Suppose that $z \not\in C \cap H^{-1}(0)$.
      Then, by \cref{lem:defined-all-times-nbhd}, the maximal trajectory of~$X_F$ that starts at $z$ is defined for all times.
    \item
      Suppose that $z \in C \cap H^{-1}(0)$.
      By \cref{lem:defined-all-times-nbhd}, the maximal trajectory of $X_F$ that starts at $z$ stays in the set $H^{-1}(0)$.
      On this set, we have $dF = \chi dH$, and so $X_F = \chi X_H$.
      By \cref{cor:max-traj-iff}, the maximal trajectory of $X_F$ starting at~$z$ coincides with the maximal trajectory of $\chi X_H$ starting at~$z$.
      By \cref{lem:reparametrize-0-1}, the domain of this trajectory contains the domain of the maximal trajectory of $X_H$ starting at $z$.
  \end{itemize}
  In either case, we obtain $S_{X_F}(z) \leq S_{X_H}(z) < 0 < T_{X_H}(z) \leq T_{X_F}(z)$.

Because $H=0$ and $\chi = 1$ on $Z$, we have $X_F=X_H$ on $Z$.
By \cref{lem:time-one-excise}, $Z$ is invariant under the forward-flow of $X_H$;
by this and \cref{cor:max-traj-iff}, $T_{X_F}=T_{X_H}$ on~$Z$.

  We now compare the forward and backward exit times of $X_F$ and $X_H$, recalling that the time-$1$ flow of $X_H$ excises $Z$ from $M$:
  \begin{itemize}
  \item On $M$, we have $S_{X_F} \leq S_{X_H} < -1$.
  \item On $M \setminus Z$, we have $T_{X_F} \geq T_{X_H} > 1$.
  \item On $Z$, we have $T_{X_F} = T_{X_H} \leq 1$.
  \end{itemize}
  By \cref{lem:time-one-excise}, we conclude that the time-$1$ flow of $X_F$ also excises $Z$ from $M$.
\end{proof}

\begin{remark}
  The following table summarizes the arguments in the proofs of \cref{lem:defined-all-times-nbhd,prop:ham-imply-nbhd}.
  
  \def\arraystretch{1.2}%
  \noindent\centering
   \begin{tabular} {|c|c|c|c|l|}
    \hline
    Subset of $M$ & Vect.\ field & Back. exit time & Forw.\ exit time & Reasoning \\ \hline \hline
    $\Set{\chi = 0}$ & $X_F = 0$ & $S_{X_F} = -\infty$ & $T_{X_F} = \infty$ & fixed points \\ \hline
    $\Set{F \neq 0}$ & $X_F$ & $S_{X_F} = -\infty$ & $T_{X_F} = \infty$ & compact level sets \\ \hline
    $\Set{\chi \neq 0, F = 0}$ & $X_F = \chi X_H$ & $S_{X_F} \leq S_{X_H} $ & $T_{X_F} \geq T_{X_H}$ & $0 \leq \chi \leq 1$ \\ \hline \hline
    $Z$ & $X_F = X_H$ & & $T_{X_F} = T_{X_H}$ & forward-invariance \\ \hline
  \end{tabular}
\end{remark}

\begin{theorem}[Locality of Hamiltonian excision] \label{thm:locality-excision}
  Let $Z_1$ and $Z_2$ be closed subsets of symplectic manifolds $(M_1, \omega_1)$ and $(M_2, \omega_2)$, respectively.
  Suppose that there exist an open neighbourhood $U_1$ of $Z_1$ in $M_1$ and an open neighbourhood $U_2$ of $Z_2$ in $M_2$ and a diffeomorphism $\varphi \colon U_1 \to U_2$ that takes $Z_1$ onto $Z_2$.
  Suppose that there exists a smooth function on $M_1$ that vanishes on $Z_1$ and whose time-$1$ Hamiltonian flow excises $Z_1$ from $M_1$.
  Then there exists a smooth function on $M_2$ that vanishes on $Z_2$ and whose time-$1$ Hamiltonian flow excises $Z_2$ from $M_2$.
  \ynote{We use \cref{thm:locality-excision} in \cref{exm:narrow nbhd}.}
\end{theorem}

\begin{proof}
  We claim that there exist subsets
  \begin{equation*}
    Z_1 \subset C_1 \subset V_1 \subset U_1 \quad \stext{and} \quad Z_2 \subset C_2 \subset V_2 \subset U_2
  \end{equation*}
  such that
  \begin{itemize}
    \item $C_1$ is closed in $M_1$, $Z_1$ is contained in the interior of $C_1$, and $V_1$ is open in $M_1$, and similarly for the sets in $M_2$.
    \item $\varphi(C_1) = C_2$, and $\varphi(V_1) = V_2$.
  \end{itemize}
  Indeed, take $V_1 := U_1 \cap \varphi^{-1}(U_2)$ and $V_2 = \varphi(U_1) \cap U_2$, let $C_1'$ be a closed neighbourhood of $Z_1$ in $M_1$ that is contained in $V_1$ and let $C_2'$ be a closed neighbourhood of $Z_2$ in $M_2$ that is contained in $V_2$, and take $C_1 := C_1' \cap \varphi^{-1}(C_2')$ and $C_2 := \varphi(C_1') \cap C_2'$.

  By \cref{prop:ham-imply-nbhd}, there exists a smooth function $H_1 \colon M_1 \to \R$ that vanishes on $Z_1$, whose support $\supp H_1$ is contained in $C_1$, and whose time-$1$ Hamiltonian flow excises $Z_1$ from $M$.
  Because the set $\varphi(\supp H_1)$ is closed in $V_2$ and contained in $C_2$, it is also closed in $M_2$.
  So the function $H_2 \colon M_2 \to \R$ that is equal to $H_1 \circ \varphi^{-1}$ on $U_2$ and $0$ outside~$U_2$ is smooth.

  Let $X_1$ and $X_2$ be the Hamlitonian vector field of $H_1$ and $H_2$.

  Because $V_1$ contains $\supp H_1$, the backward and forward exit times of the restricted vector field $\Res{X_1}_{V_1}$ are the restrictions to $V_1$ of the backward and forward exit times of~$X_1$, and similarly for $X_2$.
  By this and \cref{lem:time-one-excise}, the time-$1$ flow of $X_1$ excises $Z_1$ from $M_1$ if and only if the time-$1$ flow of $\Res{X_1}_{V_1}$ excises $Z_1$ from $V_1$, and similarly for $X_2$.

  Because $\varphi$ restricts to a diffeomorphism $V_1 \to V_2$, it intertwines the backward and forward exit times of the restricted vector fields $\Res{X_1}_{V_1}$ and $\Res{X_2}_{V_2}$.
  By this and \cref{lem:time-one-excise}, the time-$1$ flow of $\Res{X_1}_{V_1}$ excises $Z_1$ from $V_1$ if and only if the time-$1$ flow of $\Res{X_2}_{V_2}$ excises $Z_2$ from $V_2$.
  Together with the previous paragraph, this proves the theorem.
\end{proof}

\section{Flows, revisited}
\label{sec:flows}

Let $Y$ be a vector field on a manifold $N$, and let
\begin{equation*}
  \Phi_Y \colon D_Y \to N
\end{equation*}
be its maximal flow, with flow domain
\begin{equation*}
  D_Y = \Set{(z, t) \in N \times \R \mmid S_Y(z) < t < T_Y(z)}.
\end{equation*}

We will need versions of the escape lemma that are stronger than  \cref{lem:escape-1}.
Here is a well-known one:

\begin{lemma}[Escape lemma, Version 2] \label{lem:escape-2}
  For any $x \in N$, if $T_Y(x) < \infty$, then the maximal forward trajectory of $Y$ starting at $x$ escapes every compact set, and if $S_Y(x) > -\infty$, then the maximal backward trajectory of $Y$ ending at $x$ escapes every compact set.
  \ynote{We use \cref{lem:escape-3} in the proof of \cref{lem:slow-zone}.}
\end{lemma}

And here is stronger one:

\begin{lemma}[Escape lemma, Version 3] \label{lem:escape-3}
Consider the fibrewise closure of $D_Y$:
\begin{equation*}
  \wt{D}_Y := \Set{(z, t) \in N \times \R \mmid S_Y(z) \leq t \leq T_Y(z)}.
\end{equation*}
Let $\ol{N} := N \cup \Set{\infty}$ be the one-point compactification of $N$, and let
\begin{equation*}
  \wt{\Phi}_Y \colon \wt{D}_Y \to \ol{N}
\end{equation*}
be the map whose restriction to $D_Y$ is $\Phi_Y$
and that sends $\wt{D}_Y \setminus D_Y$ to $\infty$.
Then $\wt{\Phi}_Y$ is continuous.
\ynote{We use \cref{lem:escape-3} in the proof of \cref{cor:D-plus}.}
\end{lemma}

\begin{proof}[Proof of \cref{lem:escape-2,lem:escape-3}]
  We will prove \cref{lem:escape-3}, which implies \cref{lem:escape-2}.

  Suppose the contrary; then there is a sequence $((z_i, t_i))_{i = 1}^\infty \subset D_Y$ converging to some $(z_\infty, t_\infty) \in \wt{D}_Y \setminus D_Y$ while $(\Phi_Y(z_i, t_i))_{i = 1}^\infty$ is contained in a compact set $K_1$ which is itself contained in the interior of a compact $K \subset N$.
  Without loss of generality, we assume $t_\infty = T_Y(z_\infty) > 0$.
  By passing to a subsequence, we assume that $\lim_{i \to \infty} \Phi_Y(z_i, t_i) = y_1 \in K_1$ and that $t_i > 0$ for all $i$.
  By the continuity of $\Phi_Y$ at $(y_1, 0)$, there is an $N \in \N$ and $\eps_1 \in (0, t_\infty)$ such that $\Phi_Y(z_i, t) = \Phi_Y(\Phi_Y(z_i, t_i), t - t_i)\in K$ for any $t \in (t_i - \eps_1, t_i)$ and $i \geq N$.
  On the other hand, by the escape lemma (\cref{lem:escape-1}) there is an $\eps_2 \in (0, \eps_1)$ such that $y_2 := \Phi_Y(z_\infty, t_\infty - \eps_2) \in N \setminus K$.
  But $y_2 = \lim_{i \to \infty} \Phi_Y(z_i, t_i - \eps_2)$, and $\Phi_Y(z_i, t_i - \eps_2) \in K$ for all $i \geq N$, so $y_2 \in K$, giving a contradiction.
\end{proof}

\begin{corollary} \label{cor:D-plus}
  Let
  \begin{multline*}
    \D_Y^+ := \big\{(z_0, z, t) \in N \times N \times [0, \infty) \mid \\
    \stext{there exists a trajectory} \gamma \colon [0, t] \to N \stext{of} Y \stext{with} \gamma(0) = z_0 \stext{and} \gamma(t) = z \big\}.
  \end{multline*}
  Then the following holds.

  \begin{itemize}[itemsep = 6pt]
    \item[(i)]
      For any $(x_0, x,t) \in \D_Y^+$, the backward exit time $S_Y$ is continuous at $x_0$ if and only if it's continuous at $x$, and the forward exit time $T_Y$ is continuous at $x_0$ if and only if it's continuous at $x$.
    \item[(ii)]
      For any point $(x_0, x, t_\infty)$ in $N \times N \times [0, \infty)$, if $(x_0, x, t_\infty)$ is an accumulation point of $\D_Y^+$, and if the backward exit time $S_Y$ is continuous at $x$ or at~$x_0$ or the forward exit time $T_Y$ is continuous at $x$ or at~$x_0$, then $(x_0, x, t_\infty)$ is in $\D_Y^+$.
  \end{itemize}
  \ynote{We use \cref{cor:D-plus} in the proof of \cref{lem:slow-zone}.}
\end{corollary}

\begin{remark} \label{rem:cont-infty}
  Because $T_Y \colon N \to (0, \infty]$ and $S_Y \colon N \to [-\infty, 0)$ are, respectively, lower-semi-continuous and upper-semi-continuous, $T_Y$ is automatically continuous at any point where $T_Y = \infty$, and $S_Y$ is automatically continuous at any point where $S_Y = \infty$.
\eor
\end{remark}

\begin{proof}[Proof of \cref{cor:D-plus}]
  Let $(x_0, x, t) \in \D_Y^+$.
  Then $x_0$ is contained in the domain of the time-$t$ flow $\Phi_Y(\cdot, t)$ and $x$ is contained in the image of the time-$t$ flow $\Phi_Y(\cdot, t)$.
  Because the domain and the image of the time-$t$ flow are open subsets of $N$, and the time-$t$ flow is a diffeomorphism between them, in particular $\Phi_Y(\cdot, t)$ is a diffeomorphism from an open neighbourhood $U_0$ of $x_0$ onto an open neighbourhood $U$ of $x$ (see~\cref{eq:Phi-t}).
  On $U_0$, we have $S_Y \circ \Phi_Y(\cdot, t) = S_Y(\cdot) - t$ and $T_Y \circ \Phi_Y(\cdot, t) = T_Y(\cdot) - t$ (see~\cref{eq:exit-time-shift}).
  This implies item (i).

  Let $\Set{(z_{0, j}, z_j, t_j)}_{j = 1}^\infty$ be a sequence in $\D_Y^+$ that converges to a point $(x_0, x, t_\infty) \in N \times N \times [0, \infty)$ with $S_Y$ continuous at $x$.
  Then $S_Y(z_j)$ converges to $S_Y(x)$.
  Because $S_Y(z_j) < -t_j  \leq 0$ for all $j$, in the limit we get $S_Y(x) \leq -t_\infty \leq 0$.
  If $S_Y(x) < -t_\infty$, the continuity of the flow at $(x, -t_\infty)$ implies that $x_0 = \Phi_Y(x, -t_\infty)$, and so $(x_0, x, t_\infty) \in \D_Y^+$.
  Otherwise, $S_Y(x) = -t_\infty$; we claim that this possibility never occurs.
  Indeed, if $S_Y(x) = -t_\infty$, then by \cref{lem:escape-3}, we have $\Phi_Y(z_j, -t_j)\to \infty$, so $z_{0, j} \to \infty$, whereas we assumed that $z_{0, j} \to x_0$.

  Similarly, let $\Set{(z_{0, j}, z_j, t_j)}_{j = 1}^\infty$ be a sequence in $\D_Y^+$ that converges to a point $(x_0, x, t_\infty) \in N \times N \times [0, \infty)$ with $T_Y$ continuous at $x_0$.
  Then $T_Y(z_{0, j})$ converges to $T_Y(x_0)$.
  Because $0 \leq t_j < T_Y(z_{j, 0})$ for all $j$, in the limit we get $0 \leq t_\infty \leq T_Y(x_0)$.
  If $t_\infty < T_Y(x_0)$, the continuity of the flow at $(x_0, t_\infty)$ implies that $x = \Phi_Y(x_0, t_\infty)$, and so $(x_0, x, t_\infty) \in \D_Y^+$.
  Otherwise, $T_Y(x_0) = t_\infty$; we claim that this possibility never occurs.
  Indeed, if $T_Y(x_0) = t_\infty$, then by \cref{lem:escape-3}, we have $\Phi_Y(z_{0, j}, t_j)\to \infty$, so $z_j \to \infty$, whereas we assumed that $z_j \to x$.
  This proves (ii).
\end{proof}

The following lemma will allow us to modify a given a Hamiltonian vector field to obtain a new Hamiltonian vector field whose trajectories outside a given closed invariant set are defined for all times.

\begin{lemma} \label{lem:slow-zone}
  Let $Y$ be a vector field on a manifold $N$, and let $Q \subset N$ be a closed subset.
  Suppose that at each point of $Q$ the backward or forward exit time is continuous.
  If the complement of $Q$ is invariant under the forward-flow of $Y$ (equivalently, $Q$ is invariant under the backward-flow of $Y$), then $Q$ has a neighbourhood $V$ in $N$ such that any maximal forward trajectory of $Y$ that is contained in $V \setminus Q$ is defined for all positive times.
  \ynote{We use \cref{lem:slow-zone} in the proof of \cref{lem:extension-vector-field-germ}.}
\end{lemma}

\begin{proof}
  Because $Q$ is closed in $N$, there exists a smooth function $\rho \colon N \to [0, 1]$ whose zero locus is $Q$; see, e.g., \cite[Theorem~2.29]{MR2954043}.
  Let $\zeta \colon N \to [0, \infty)$ be an exhaustion function for $N$.
  Let $\D_Y^+$ be as in \cref{cor:D-plus}, and let
  \begin{equation*}
    \E_Y := \Set{(z_0, z, t) \in \D_Y^+ \mmid \zeta(z) - \zeta(z_0) \geq t \stext{and} \rho(z_0)(\zeta(z) - \zeta(z_0)) = 1}.
  \end{equation*}
  Note that
  \begin{equation} \label{eq:bounded}
    \stext{for all} (z_0, z, t) \in \E_Y, \quad t \stext{and} \zeta(z_0) \stext{are both in} [0, \zeta(z)].
  \end{equation}
  Let $E_Y$ be the image of $\E_Y$ under the projection map $(z_0, z, t) \mapsto z$.
  We claim that the closure of $E_Y$ is disjoint from $Q$.
  Indeed, assuming otherwise, let $(z_{0, j}, z_j, t_j)$ be a sequence in $\E_Y$ with $\lim_{j \to \infty} z_j = x \in Q$.
  By \cref{eq:bounded}, and since $\zeta(z_j) \to \zeta(x)$, after passing to a subsequence we may assume that $t_j \to t_\infty$ and $z_{0, j} \to x_0$.
  By \cref{cor:D-plus}, $(x_0, x, t_\infty) \in \D_Y^+$.
  Finally, $1 = \rho(z_{0, j}) \big(\zeta(z_j) - \zeta(z_{0, j}) \big) \to \rho(x_0) \big(\zeta(x) - \zeta(x_0) \big)$, so $\rho(x_0) \neq 0$, and so $x_0 \not \in Q$.
  Because the complement of $Q$ is invariant under the forward flow of $Y$ and $x = \Phi_Y(x_0, t_\infty)$, this contradicts our assumption that $x \in Q$.

  Then $V := N \setminus E_Y$ is a neighbourhood of $Q$ in $N$.
  Let $\gamma$ be a maximal forward trajectory of $Y$ that is contained in $V \setminus Q$.
  We claim that $\gamma$ is defined for all positive times.
  Assuming otherwise, let $z_0 \in \gamma$ be such that $T_Y(z_0) < 1$.
  By the escape lemma (\cref{lem:escape-2}) and the intermediate value theorem, there exists $t \in (0,T_Y(z_0))$ such that $z := \Phi_Y(z_0, t)$ satisfies $\zeta(z) = \zeta(z_0) + \rho(z_0)^{-1}$.
  Note that $\zeta(z) - \zeta(z_0) \geq 1 > T_Y(z_0) > t$.
  Then $(z_0, z, t) \in \E_Y$.
  So $z \in E_Y$, which contradicts the assumption that $\gamma$ is contained in $V$.
\end{proof}

\section{From smooth excision to Hamiltonian excision}
\label{sec:smooth-to-hamiltonian}

In this section, we extend an excision by a null vector field on a submanifold into a Hamiltonian excision on the ambient symplectic manifold.
We begin by extending a null vector field to a Hamiltonian vector field:

\begin{lemma} \label{lem:extension-vector-field-singular}
  Let $N$ be an embedded submanifold of a symplectic manifold $(M, \omega)$, and let $Y$ be a null vector field on $N$.
  Then there exist an open neighbourhood $U$ of $N$ in~$M$ and a smooth function $H \colon U \to \R$, such that $N$ is closed in $U$, the function $H$ vanishes on $N$, and the Hamiltonian vector field $X_H$ of $H$ coincides with $Y$ on $N$.
  \ynote{ \\ We use \cref{lem:extension-vector-field-singular} in the proof of \cref{thm:extension-vector-field-germ}.}
\end{lemma}

\begin{proof}
  Let $\omega_N$ denote the pullback of $\omega$ to $N$.
  By assumption, $Y \contract \omega_N = 0$.
  Let $T_{M/N}$ and $T^*_{M/N}$ denote the normal and conormal bundles to $N$ in $M$.
  Then $Y \in \Gamma((TN)^{\omega_N})$, and
  \begin{equation*}
    \alpha := Y \intprod \omega \in \Gamma(T^*_{M/N}).
  \end{equation*}
  Viewing $\alpha$ as a fibrewise linear function on $T_{M/N}$, and by the tubular neighbourhood theorem, there exist a tubular neighbourhood $U$ of $N$ in $M$ and a smooth function $H \colon U \to \R$, such that
  \begin{equation*}
    \Res{H}_N = 0 \quad \stext{and} \quad \Res{\der H}_N = \alpha.
  \end{equation*}
  This function is as required.
\end{proof}

In the setup of \cref{lem:extension-vector-field-singular}, let $Z$ be a closed subset of $N$, and suppose that the time-$1$ flow of the vector field $X$ excises $Z$ from~$N$.
Suppose, moreover, that $Z$ is closed not only in $N$ but also in the ambient manifold $M$.
Then we can multiply the function $H$ by a cutoff function $M \to [0, 1]$ that is supported in $U$ and is equal to $1$ near $Z$.
This yields a smooth function whose time-$1$ Hamiltonian flow is a diffeomorphism from a subset of $M \setminus Z$ that contains $N \setminus Z$ to a subset of $M$ that contains $N$.
We would like to choose the cutoff function such that the domain of this diffeomorphism will be all of $M \setminus Z$ and its image will be all of~$M$.

We will use the following technical lemma, which relies on \cref{lem:slow-zone} and which applies to Hamiltonian flows.

\begin{lemma} \label{lem:extension-vector-field-germ}
  Let $(M, \omega)$ be a symplectic manifold, $Z$ a closed subset of $M$, $U$ an open neighbourhood of $Z$ in $M$, and $H \colon U \to \R$ a smooth function that vanishes on $Z$.
  Consider the flow of the Hamiltonian vector field $X_H$ of $H$, let $Q$ be the closure in $U$ of the set of points whose trajectories pass through~$Z$,
  and suppose that the backward exit time $S_{X_H}$ is equal to $-\infty$ on $Q$.
  Then there is a closed neighbourhood $C$ of $Z$ in $M$ that is contained in $U$, such that the following holds.

  Let $\chi \colon M \to [0, 1]$ be any smooth function that is supported in $C$ and is equal to $1$ on~$Z$, and such that $d\chi$ vanishes whenever $\chi$ vanishes.
  Define $F: = \chi H$ on $U$ and $F=0$ outside the support of $\chi$.
  Let $X_F$ be its Hamiltonian vector field.
  Then any maximal trajectory of $X_F$ starting in $M \setminus Q$ is defined for all times.
  \ynote{ \\ We use \cref{lem:extension-vector-field-germ} in the proof of \cref{thm:extension-vector-field-germ}.}
\end{lemma}

\begin{proof}
The assumptions of the lemma imply that the backward exit time $S_{X_H} \colon U \to [-\infty, 0)$ is continuous at the points of $Q$ (see \cref{rem:cont-infty}) and that $Q$ is invariant under the flow of $X_H$ on $U$.
Applying \cref{lem:slow-zone} to $X_H$ and to $-X_H$ on $U$, we find a neighbourhood $V$ of $Q$ in $U$ such that every maximal forward trajectory of $X_H$ that is contained in $V \setminus Q$ is defined for all positive times and every maximal backward trajectory of $X_H$ that is contained in $V \setminus Q$ is defined for all negative times.

By \cref{lem:vanish-at-infinity-in-nbhd-v1}, we find a closed neighbourhood $C$ of $Z$ in $M$, contained in $V$, such that $\Res{H}_{C}$ vanishes at infinity.
We claim that $C$ has the required properties.

Let $\chi$ and $F$ be as in the statement of the lemma.

Because $d\chi$ vanishes whenever $\chi$ vanishes, the support of $X_F$ is contained in the support of $\chi$, which is contained in $U$.
So $M \setminus U$ is fixed under the flow of $X_F$, and $U$ is invariant under the flow of $X_F$ (see \cref{cor:invariant-submfd}).
So we may now focus on the flow of $\Res{X_F}_U$ on~$U$.
It remains to show that, for each $z \in U \setminus Q$, the maximal trajectory of $\Res{X_F}_U$ that starts at $z$ is defined for all times.
Fix such a $z$.

\begin{itemize} [itemsep = 6pt]
  \item
    Suppose that $z \not\in C \cap H^{-1}(0)$.
    By \cref{lem:defined-all-times-nbhd} for the manifold $U$, the maximal trajectory of $\Res{X_F}_U$ that starts at $z$ is defined for all times.
  \item
    Suppose that $z \in C \cap H^{-1}(0)$.
    By \cref{lem:defined-all-times-nbhd} for the manifold $U$, the maximal trajectory of $\Res{X_F}_U$ that starts at $z$ stays in the set $C \cap H^{-1}(0)$.
    On this set, we have $H=0$, so $\Res{X_F}_U = \chi X_H$.
    By \cref{cor:max-traj-iff}, this maximal trajectory coincides with the maximal trajectory of $\chi X_H$ starting at $z$.
    By \cref{lem:reparametrize-0-1}, the domain of this trajectory contains the domain of the maximal trajectory of $X_H$ starting at $z$.
    Because the set $C \cap H^{-1}(0)$ is contained in~$V$ and by the choice of $V$, and since $z \not\in Q$, the maximal trajectory of $X_H$ starting at $z$ is defined for all times.
    It follows that the maximal trajectory of $\Res{X_F}_U$ starting at $z$ is also defined for all times.
  \end{itemize}
\end{proof}

Combining \cref{lem:extension-vector-field-germ} with \cref{lem:extension-vector-field-singular}, we show how to get from a smooth excision by a null vector field on a submanifold to a Hamiltonian excision.

\begin{theorem} \label{thm:extension-vector-field-germ}
  Let $Z$ be a closed subset of a symplectic manifold $(M, \omega)$.
  Suppose that there exist an embedded submanifold $N$ of $M$ that contains $Z$ and a null vector field $Y$ on~$N$ whose time-$1$ flow excises $Z$ from $N$.
  Then there exists a smooth function $F \colon M \to \R$ whose time-$1$ Hamiltonian flow excises $Z$ from~$M$.
  \ynote{ \\ We use \cref{thm:extension-vector-field-germ} in the proofs of \cref{cor:hamiltonian-excision-new,prop:exc-epigraph,cor:epigraph-hamiltonian-excision}.}
\end{theorem}

\begin{proof}
  Let $N$ and $Y$ be as in the statement of the theorem.
  Denote by $\omega_N$ the pullback of $\omega$ to $N$.
  By \cref{lem:time-one-excise},
  \begin{equation*}
    S_Y = - \infty \stext{on} N, \quad T_Y > 1 \stext{on} N \setminus Z, \quad \stext{and} \quad T_Y \leq 1 \stext{on} Z.
  \end{equation*}

  By \cref{lem:extension-vector-field-singular}, we find an open neighbourhood $U$ of $N$ in $M$ and a smooth function $H \colon U \to \R$ such that $N$ is closed in $U$, $H$ vanishes on $N$, and the Hamiltonian vector field $X_H$ of $H$ coincides with $Y$ on $N$.
  By \cref{cor:invariant-submfd}, the subset $N$ of $U$ is invariant under the flow of $X_H$.
  In particular, every maximal trajectory of $Y$ in $N$ is also a maximal trajectory of $X_H$ in $U$.
  So $S_{X_H}|_N = S_Y$ and $T_{X_H}|_N = T_Y$, and so
  \begin{equation*}
    S_{X_H} = - \infty \stext{on} N, \quad T_{X_H} > 1 \stext{on} N \setminus Z, \quad \stext{and} \quad T_{X_H} \leq 1 \stext{on} Z.
  \end{equation*}

  Now consider the flow of $X_H$ on $U$, and let $Q$ be the closure in $U$ of the set of points whose trajectories pass through $Z$.
  Because $N$ is invariant under the flow of $X_H$ and is closed in $U$, the set $Q$ is contained in $N$, so $S_{X_H} = -\infty$ on $Q$.
  Applying \cref{lem:extension-vector-field-germ}, we find a closed neighbourhood $C$ of $Z$ in $M$ that is contained in $U$ and that has the properties described in \cref{lem:extension-vector-field-germ}.
  Let $\chi \colon M \to [0, 1]$ and $F := \chi H$ be as in \cref{lem:extension-vector-field-germ}, so that any maximal trajectory of $X_F$ starting in $M \setminus Q$ is defined for all times.
  Then
  \begin{equation*}
    S_{X_F} = -\infty \quad \stext{and} \quad T_{X_F} = \infty \qquad \stext{(on $M \setminus Q$, hence) on} M \setminus N.
  \end{equation*}

  On the set $N$, the function $H$ vanishes, so $\Res{X_F}_N = \chi X_H|_N = \chi Y$.
  By \cref{lem:cutoff-smooth}, the time-$1$ flow of $\chi Y$ on $N$ excises $Z$ from $N$.
  So the time-$1$ flow of $\Res{X_F}_N$ on $N$ excises $Z$ from $N$.

  Because $C$ is contained in $U$ and $N$ is closed in $U$, the intersection $C \cap N$ is closed in~$C$.
  Because $C$ is closed in $M$, it follows that $C \cap N$ is closed in $M$.
  Because the support of $X_F$ is contained in $C$ and by \cref{cor:invariant-submfd}, $N$ is invariant under the flow of $X_F$.
  In particular, every maximal trajectory of $\Res{X_F}_N$ in $N$ is also a maximal trajectory of $X_F$ in $M$.
  Because the time-$1$ flow of $\Res{X_F}_N$ on $N$ excises $Z$ from $N$,
  \begin{equation*}
    S_{X_F} = - \infty \stext{on} N, \quad T_{X_F} > 1 \stext{on} N \setminus Z, \quad \stext{and} \quad T_{X_F} \leq 1 \stext{on} Z.
  \end{equation*}
  From this analysis of the backward and forward exit times of the flow of $X_F$ on $M \setminus N$, on $N$, on $N \setminus Z$, and on $Z$, we conclude that the time-$1$ flow of $X_F$ excises $Z$ from~$M$.
\end{proof}

\begin{corollary} \label{cor:hamiltonian-excision-new}
  Let $N_0$ be a manifold, $Z_0$ a closed subset of $N_0$, and $Y_0$ a vector field on $N_0$ whose time-$1$ flow excises $Z_0$ from $N_0$.
  Let $(M, \omega)$ be a symplectic manifold, let $U_0$ be an open neighbourhood of $Z_0$ in $N_0$, and let $i \colon U_0 \to M$ be an embedding whose restriction to $Z_0$ is proper and such that $Y_0 \contract i^* \omega = 0$.
  Then there exists a smooth function $F \colon M \to \R$ whose time-$1$ Hamiltonian flow excises $Z := i(Z_0)$ from $M$.
  \ynote{ \\ We use \cref{cor:hamiltonian-excision-new} in the proof of \cref{thm:stratmann}.}
\end{corollary}

\begin{proof}
  By \cref{cor:cutoff-smooth}, we may assume that the support of $Y_0$ is contained in $U_0$.
  Let $N := i(U_0)$; then $N$ is an embedded submanifold of $M$ that contains $Z$, and $i \colon U_0 \to N$ is a diffeomorphism.
  Let $X := i_* Y_0$ be the vector field on $N$ such that $Y_0$ and $X$ are $i$-related; then $X$ is a null vector field on $N$ whose time-$1$ flow excises $Z$ from $N$.
  Because the restriction of $i$ to $Z_0$ is proper, $Z$ is closed in $M$.
  The result then follows from \cref{thm:extension-vector-field-germ}.
\end{proof}

We now recover Stratmann's result.
Here is an equivalent statement:

\begin{theorem}[Stratmann \cite{MR4537129}] \label{thm:stratmann}
  Let $\varphi \colon [-1, \infty) \times \Sigma \to M$ be a proper injective immersion from the manifold-with-boundary $[-1, \infty) \times \Sigma$, where $\Sigma$ is some manifold, to a symplectic manifold $(M, \omega)$.
  Suppose that $\frac{\partial}{\partial t} \contract \varphi^* \omega = 0$, where $\frac{\del}{\del t}$ is the standard vector field on $[-1, \infty)$.
  Then $M$ and $M \setminus \varphi([0, \infty) \times \Sigma)$ are symplectomorphic.
  Moreover, for any neighbourhood of $\varphi([0, \infty) \times \Sigma)$ in $M$, there exists a symplectomorphism $M \to M \setminus \varphi([0, \infty) \times \Sigma)$ that is the identity outside the given neighbourhood.
\end{theorem}

\begin{proof}
  Let $Z_0 := [0, \infty) \times \Sigma$.
  Consider the manifold $N_0 := (-1, \infty) \times \Sigma$, denote by $t$ the coordinate on its first factor, and let $X_0$ be the vector field on $N_0$ that is equal to the product of $(t+1)^2 \frac{\partial}{\partial t}$ with some smooth function that takes values in $[0, 1]$, is equal to $1$ when $t \geq 0$, and vanishes when $t \leq -\frac12$.
  By an explicit calculation of the flow of $(t+1)^2 \frac{\partial}{\partial t}$ on $(-1, \infty)$, we conclude that the time-$1$ flow of $X_0$ excises $Z_0$ from $N_0$.
  Because the map $\varphi \colon [-1, \infty) \times \Sigma \to M$ is proper, its restriction to the closed subset $Z_0$ is proper.
  The result then follows from \cref{cor:hamiltonian-excision-new}, with $U_0 := N_0$ and $i := \Res{\varphi}_{N_0}$, and from \cref{prop:ham-imply-nbhd}.
\end{proof}

\section{Smooth excision of epigraphs}
\label{sec:smooth-epigraph-excision}

Let $\Sigma$ be a manifold, and fix a function $\lambda \colon \Sigma \to (0, 1]$.
We look for a vector field on $\Sigma \times (0, 1)$ of the form
\begin{equation*}
  v(p, x) \deldel{x}
\end{equation*}
whose time-$1$ flow excises the subset.
\begin{equation} \label{eq:Z}
  Z := \Set{(p, x) \in \Sigma \times (0, 1) \mmid x \geq \lambda(p)}
\end{equation}
from $\Sigma \times (0, 1)$.
Strictly speaking, the set $Z$ is not quite an epigraph, because $\lambda$ takes values in $(0, 1]$, whereas $Z$ is a subset of $\Sigma \times (0, 1)$, not $\Sigma \times (0, 1]$.
Nevertheless, as in \cref{lem:semi-cont}, here too a necessary condition is that $\lambda$ be lower semi-continuous:

\begin{lemma} \label{lem:Z-closed}
  If there exists a vector field on $\Sigma \times (0, 1)$ whose time-$1$ flow excises the subset~\cref{eq:Z}, then the function $\lambda \colon \Sigma \to (0, 1]$ is lower semi-continuous.\footnote{\label{fn:to-J}By \cref{lem:semicontinuous-to-interval}, it does not matter if we view $\lambda$ as a function to $(0, 1]$ or as a function to $\R$.}
  \ynote{ \\ We use \cref{lem:Z-closed} in the text above \cref{thm:epigraph-smooth-excision}.}
\end{lemma}

\begin{proof}
  Suppose that there exists a vector field on $\Sigma \times (0, 1)$ whose time-$1$ flow excises the above subset $Z$.
  By \cref{fn:topology}, $Z$ must be closed in $\Sigma \times (0, 1)$.
  Because the complement of $Z$ in $\Sigma \times (0, 1)$ coincides with the complement of the union $Z \cup (\Sigma \times \Set{1})$ in $\Sigma \times (0, 1]$, the union $Z \cup (\Sigma \times \Set{1})$ is closed in $\Sigma \times (0, 1]$.
  But this union is exactly the epigraph of $\lambda \colon \Sigma \to (0, 1]$ in $\Sigma \times (0, 1]$.
  Because this epigraph is closed, by \cref{lem:semi-cont}, the function $\lambda \colon \Sigma \to (0, 1]$ is lower semi-continuous.
\end{proof}

The main result of this section is that the lower semi-continuity of the function $\lambda$ is not only necessary but is also sufficient for excisability of the subset~\cref{eq:Z}.
See \cref{thm:epigraph-smooth-excision}, preceded by the easier result \cref{lem:epigraph-smooth-forward-time}.

\subsection*{The one dimensional case}

We recall the solution theory of an autonomous ordinary differential equation of first order on the interval $(0, 1)$.
Consider a non-negative smooth function $v \colon (0, 1) \to [0, \infty)$ and the vector field $X = v(x) \deldel{x}$.
Its forward exit time is found by Barrow's formula,
\begin{equation*}
  T_X(x) = \begin{dcases}
    \int_x^1 \frac{\der \xi}{v(\xi)} \in (0, \infty] & \stext{if}v(\xi) > 0 \stext{for all} \xi \in [x, 1), \\
    \infty & \stext{if}v(\xi) = 0 \stext{for some} \xi \in [x, 1),
  \end{dcases}
\end{equation*}
and similarly for its backward exit time $S_X$.
In particular, if $v(x) = 0$ or all $x$ close to the left endpoint $0$, then the image of $S_X$ is $\Set{-\infty}$ and the image of $T_X$ is either $\Set{\infty}$ or $(0, \infty]$.
In the latter case, there exists $x_0 \in (0, 1)$ such that $T_X$ is equal to $\infty$ on $(0, x_0]$ and is strictly decreasing on $[x_0, 1)$, and the time-$1$ flow of $X$ excises $[\lambda_0, 1)$ from $(0, 1)$ where $T_X^{-1}(\Set{1}) = \Set{\lambda_0}$.

\subsection*{Epigraphs of smooth functions on closed subsets}

When $\lambda \colon \Sigma \to (0, 1]$ is smooth, excising $Z = \Set{(p, x) \mmid  x \geq \lambda(p)}$ from $\Sigma \times (0, 1)$ is not difficult.
However, this case is too restrictive even for excising rays.
In this subsection we address the more general case $Z = \Set{(p, x) \mmid  x \in C \stext{and} x \geq \lambda_C(p)}$ where $\lambda_C \colon C \to (0, 1]$ is a smooth function on is an arbitrary closed subset $C$ of $\Sigma$.
See \cref{lem:epigraph-smooth-forward-time} and \cref{exm:ray}.

We begin by constructing a model with parameters:

\begin{lemma} \label{lem:one-dim-param}
  Let
  \begin{equation*}
    \Sigma_0 = \Set{(b, c) \mmid 0 < b \leq 1 \stext{and} c \geq 0}.
  \end{equation*}
  There exists a smooth function $u \colon \Sigma_0 \times (0, 1) \to [0, \infty)$ such that the vector field $Y = u(b, c; x) \deldel{x}$ on $\Sigma_0 \times (0, 1)$ has the following properties.
  \begin{itemize}
    \item
      The backward exit time of $Y$ is always $-\infty$:
      \begin{equation} \label{eq:Su-minus-infty}
        S_Y(b, c; x) = -\infty \quad \stext{for all} \quad (b, c; x) \in \Sigma_0 \times (0, 1).
      \end{equation}
    \item
      The forward exit time of $Y$ satisfies
      \begin{equation} \label{eq:Tu-leq-one}
        T_Y(b, c; x) \leq 1 \quad \stext{if and only if} \quad c = 0 \stext{and} x \geq b.
      \end{equation}
  \end{itemize}
  \ynote{We use \cref{lem:one-dim-param} in the proof of \cref{lem:epigraph-smooth-forward-time}.}
\end{lemma}

\begin{proof}
  Define
  \begin{equation} \label{eq:vec-par}
    u(b, c; x) \coloneqq \chi(b; x) (1-b) \frac{1 - x^2}{1 - x^2 + c},
  \end{equation}
  where
  \begin{equation*}
  \begin{split}
    \chi(b; x) &\coloneqq \begin{dcases}
      0 & \stext{if} 0 < x \leq \frac{b}{2}; \\
      \frac{\exp(-\frac{1}{x-b/2})}{\exp(-\frac{1}{x-b/2}) + \exp(-\frac{1}{b-x})} & \stext{if} \frac{b}{2} < x < b; \\
      1 & \stext{if} b \leq x < 1.
    \end{dcases}
  \end{split}
  \end{equation*}
  Because the vector field $Y := u \deldel{x}$ on $\Sigma_0 \times (0, 1)$ is a non-negative multiple of $\deldel{x}$ that vanishes when $0 < x \leq \frac{b}{2}$,
  \begin{equation*}
    S_Y(b, c; x) = -\infty \quad \stext{for all} 0 < b \leq 1, c \geq 0, \stext{and} 0 < x < 1.
  \end{equation*}
  Because $u(b, c; x) = 0$ if $b = 1$ or $0 < x \leq \frac{b}{2}$,
  \begin{equation*}
    T_Y(b, c; x) = \infty \quad \stext{if} b = 1 \stext{or} 0 < x \leq \frac{b}{2}.
  \end{equation*}
  Otherwise, $u(b, c;\xi) > 0$, and
  \begin{equation} \label{eq:one-dim-param-T}
    T_Y(b, c; x) = \int_x^1 \frac{\der \xi}{u(b, c; \xi)} \quad \stext{if} 0 < b < 1 \stext{and} \frac{b}{2} < x < 1.
  \end{equation}
  For $b \leq x <1$,~\cref{eq:vec-par} and~\cref{eq:one-dim-param-T} give
  \begin{multline*}
    T_Y(b, c; x) = \frac{1}{1-b} \int_x^1 \Pa{1 + \frac{c}{1-\xi^2}} \der \xi = \frac{1}{1-b} \left[\xi + \frac{c}{2} \ln\abs{\frac{1+\xi}{1-\xi}}\right]_x^1 = \begin{dcases}
      \frac{1-x}{1-b}, & c = 0 ; \\
      \infty, & c > 0.
    \end{dcases}
  \end{multline*}
  In particular,
  \begin{equation*}
    T_Y(b, c; x) \begin{cases}
      \ > 1 \ & \stext{if} c>0 \stext{and} b \leq x < 1, \\
      \ < 1 \ & \stext{if} c = 0 \stext{and} b < x < 1, \\
      \ = 1 \ & \stext{if} c = 0 \stext{and} x = b < 1.
    \end{cases}
  \end{equation*}
  For $\frac{b}{2} < x < b < 1$,~\cref{eq:one-dim-param-T} gives
  \begin{equation} \label{additive}
    T_Y(b, c; x) = \int_x^b \frac{d \xi}{ u(b, c;\xi) } + T_Y(b, c; b).
  \end{equation}
  Because the integrand in~\cref{additive} is positive and $T_Y(b, c; b) \geq 1$,
  \begin{equation*}
    T_Y(b, c; x) > 1 \quad \stext{if} \frac{b}{2} < x < b < 1.
  \end{equation*}
\end{proof}

\begin{proposition} \label{lem:epigraph-smooth-forward-time}
  Let $\Sigma$ be a smooth manifold, let $C \subseteq \Sigma$ be a closed subset, and let $\lambda_C \colon C \to (0, 1]$ be a smooth function.
  Then there exists a vector field $X$ on $\Sigma \times (0, 1)$, of the form $X = v(p, x) \deldel{x}$ with $0 \leq v(p, x) \leq 1$, whose time-$1$ flow excises the subset
  \begin{equation*}
    \Set{(p, x) \mmid p \stext{in} C \stext{and} \lambda_C(p) \leq x < 1}
  \end{equation*}
  from $\Sigma \times (0, 1)$.
  \ynote{We use \cref{lem:epigraph-smooth-forward-time} in the proof of \cref{prop:exc-epigraph}.}
\end{proposition}

\begin{proof}
  By \cref{lem:time-one-excise}, it is enough to find a vector field of the required form whose forward exit time $T_X$ and backward exit time $S_X$ satisfy
  \begin{equation} \label{eq:excise-condition-T}
    T_X(p, x) \leq 1 \quad \stext{if and only if} \quad p \in C \stext{and} \lambda_C(p) \leq x < 1
  \end{equation}
  and
  \begin{equation} \label{eq:excise-condition-S}
    S_X(p, x) = -\infty \quad \stext{for all} \quad (p, x) \in \Sigma \times (0, 1).
  \end{equation}

  By the definition of a smooth function on a subset of a manifold, and because the subset $C$ of $\Sigma$ is closed, $\lambda_C$ is the restriction to $C$ of a smooth function $b \colon \Sigma \to (0, 1]$.
  Because $C$ is closed, it is the zero locus of a smooth function $c \colon \Sigma \to [0, 1]$; see, e.g., \cite[Theorem~2.29]{MR2954043}.
  We set
  \begin{equation*}
    v(p, x) := u(b(p), c(p); x)
  \end{equation*}
  with the function $u(b, c; x)$ of \cref{lem:one-dim-param}.
  Then for all $(p, x) \in \Sigma\times (0, 1)$, we have $T_X(p, x) = T_Y(b(p), c(p); x)$ and $S_X(p, x) = S_Y(a(p), b(p), c(p); x)$.
  By \cref{eq:Tu-leq-one}, for all $(p, x) \in \Sigma \times (0, 1)$,
  \begin{equation*}
    T_X(p, x) \leq 1 \quad \stext{if and only if} \quad c(p) = 0 \stext{and} b(p) \leq x < 1.
  \end{equation*}
  Because $c(p) = 0$ iff $p \in C$, and because for all such $p$ we have $b(p) = \lambda_C(p)$, we obtain~\cref{eq:excise-condition-T}.
  By \cref{eq:Su-minus-infty}, we obtain~\cref{eq:excise-condition-S}.
\end{proof}

\subsection*{Epigraphs of lower semi-continuous functions}

The set that we excise in \cref{lem:epigraph-smooth-forward-time} can also be expressed as
\begin{equation*}
  \Set{(p, x) \in \Sigma \times (0, 1) \mmid x \geq \lambda(p)},
\end{equation*}
where $\lambda \colon \Sigma \to (0, 1]$ is the lower semi-continuous function
\begin{equation*}
  \lambda := \begin{cases} \lambda_C & \stext{on} C \\
  1 & \stext{on} \Sigma \setminus C \end{cases}
\end{equation*}
with $\lambda_C$ smooth on $C$.
Our next goal is to excise similar sets for \emph{arbitrary} lower semi-continuous functions.
We will use the following smooth version of a theorem of Baire (see~\cite{MR1400223}).

\begin{lemma} \label{lem:smooth-Baire}
  Let $\Sigma$ be a smooth manifold, and let $\lambda \colon \Sigma \to (0, 1]$ be a lower semi-continuous function.
  Then there exists a strictly increasing sequence of positive smooth functions $f_n \colon \Sigma \to [0, 1)$, for $n \in \N$, whose supremum is~$\lambda$.
  \ynote{ \\ We use \cref{lem:smooth-Baire} in the proof of \cref{thm:epigraph-smooth-excision}.}
\end{lemma}

\begin{proof}
  Fix a metric $d(\cdot, \cdot)$ on $\Sigma$.
  We will construct a strictly increasing sequence of smooth functions
  \begin{equation*}
    f_n \colon \Sigma \to [0, 1)
  \end{equation*}
  that have the following property.
  For each $n \in \N$ and $y \in \Sigma$,
  \begin{equation} \label{eq:property-lambda} \textstyle
    \lambda(y) > f_n(y) \geq (1 - \frac{1}{n}) \inf \Set{\lambda(x) \mmid d(x, y) < \frac{1}{n}}.
  \end{equation}

  This property implies that the supremum of the sequence is $\lambda$.
  Indeed, fix $y \in \Sigma$, and let $\eps > 0$.
  Because $\lambda(y)$ is positive, the set
  \begin{equation*} \textstyle
    \Set{x \mmid \lambda(x) > (1 - \frac{\eps}{2}) \lambda(y)}
  \end{equation*}
  contains $y$; because the function $\lambda$ is lower semi-continuous, this set is a neighbourhood of $y$ in $\Sigma$.
  Let $N \in \N$ be such that $d(x, y) < \frac{1}{N}$ implies that $x$ is in this neighbourhood and such that $\frac{1}{N} < \frac{\eps}{2}$.
  Then, for any $n \geq N$,
  \begin{equation*} \textstyle
    \lambda(y) > f_n(y) \geq f_N(y) \geq (1 - \frac{1}{N}) \inf \Set{\lambda(x) \mmid d(x, y) < \frac{1}{N}} \geq (1 - \frac{\eps}{2})^2 \lambda(y) > (1 - \eps) \lambda(y).
  \end{equation*}
  Since $\eps > 0$ is arbitrary, this shows that $\lim\limits_{n \to \infty} f_n(y) = \lambda(y)$.

  To construct such a sequence $\{f_n\}$, we start with $n = 0$ and $f_0 \equiv 0$, and we continue recursively.
  Let $n \geq 1$, and let $f_{n-1} \colon \Sigma \to [0, 1)$ be a non-negative smooth function such that $\lambda(y) > f_{n-1}(y)$ for all $y \in \Sigma$.
  We will construct a smooth function $f_n \colon \Sigma \to [0, 1)$ that satisfies \cref{eq:property-lambda} and such that $f_n(y) > f_{n-1}(y)$ for all $y \in \Sigma$.

  For each $x \in \Sigma$, let $c(x, n)$ be a real number such that
  \begin{equation*} \textstyle
    \lambda(x) > c(x, n) > \max \Set{f_{n-1}(x), (1 - \frac{1}{n})\lambda(x)},
  \end{equation*}
  and let
  \begin{equation*} \textstyle
    U(x, n) := \Set{y \in \Sigma \mmid d(y, x) < \frac{1}{n} \stext{and} \lambda(y) > c(x, n) > f_{n-1}(y)}.
  \end{equation*}
  Because $\lambda(x) > c(x, n) > f_{n-1}(x)$ (by the choice of $c(x, n)$), the set $U(x, n)$ contains~$x$; because $f_{n-1}$ is (smooth, hence) continuous and $\lambda$ is lower semi-continuous, the set $U(x, n)$ is open.
  So $U(x, n)$ is a neighbourhood of~$x$ in $\Sigma$.

  For each $x \in \Sigma$ and $y \in U(x, n)$,
  \begin{equation} \label{eq:squeeze-lambda-n-c} \textstyle
    \lambda(y) > c(x, n) > f_{n-1}(y).
  \end{equation}
  Also, because $c(x, n) > (1-\frac{1}{n}) \lambda(x)$ (by the choice of $c(x, n)$) and $d(y, x) < \frac{1}{n}$ (because $y \in U(x, n)$),
  \begin{equation} \label{eq:squeeze-c-inf-lambda} \textstyle
    c(x, n) > (1 - \frac{1}{n}) \inf \Set{\lambda(z) \mmid d(y, z) < \frac{1}{n}}.
  \end{equation}

  Let $(\rho_{n, i})_{i \geq 1}$ be a smooth partition of unity with $\supp \rho_{n, i} \subseteq U(x_i, n)$ for all $i$.
  Define
  \begin{equation*}
    f_n(y) := \sum_{i \geq 1} \rho_{n, i}(y) c(x_i, n).
  \end{equation*}
  Then $f_n$ is smooth.
  For each $y \in \Sigma$, from the inequalities \cref{eq:squeeze-lambda-n-c,eq:squeeze-c-inf-lambda}, we obtain
  \begin{equation*} \textstyle
    \lambda(y) > f_n(y) > f_{n-1}(y) \quad \stext{and} \quad f_n(y) > (1 - \frac{1}{n}) \inf \Set{\lambda(z) \mmid d(y, z) < \frac{1}{n}}.
  \end{equation*}
  So $f_n > f_{n-1}$, and $f_n$ satisfies \cref{eq:property-lambda}.
\end{proof}

We now spell out an easy lemma:

\begin{lemma} \label{lem:g-between-one-f}
  Let $\Sigma$ be a manifold, and let $f \colon \Sigma \to (0, 1)$ be a continuous function.
  Then there exists a smooth function $g \colon \Sigma \to (0, 1)$ such that $f < g < 1$.
  \ynote{ \\ We use \cref{lem:g-between-one-f} in the proof of \cref{thm:epigraph-smooth-excision}.}
\end{lemma}

\begin{proof}
  For each $x \in \Sigma$, let $c(x)$ be a real number such that $f(x) < c(x) < 1$, and let $U(x) := \Set{y \in \Sigma \mmid f(y) < c(x)}$.
  By the choice of $c(x)$, the set $U(x)$ contains $x$; because $f$ is continuous, the set $U(x)$ is open.
  So $U(x)$ is a neighbourhood of $x$ in~$\Sigma$.
  Let $(\rho_i)_{i \geq 1}$ be a smooth partition of unity with $\supp \rho_i \subseteq U(x_i)$.
  Take
  \begin{equation*}
    g(y) := \sum_{i \geq 1} \rho_i(y) c(x_i).
  \end{equation*}
\end{proof}

We continue with a few preparatory lemmas.

\begin{lemma} \label{lem:u-a-b-tau}
  Let $\Sigma'_0 := \Set{(a, b, \tau) \in \R^3 \mmid 0 < a < b < 1 \stext{and} \tau > 0}$.
  There exists a smooth function
  \begin{equation*}
    u' \colon \Sigma'_0 \times (0, 1) \to (0, 1],
  \end{equation*}
  such that the following conditions hold.
  \begin{itemize}[topsep = 4pt, itemsep = 4pt]
    \item If $(a, b, \tau) \in \Sigma'_0$ and $x \notin (a, b)$, then $u'(a, b, \tau; x) = 1$.
    \item For each $(a, b, \tau) \in \Sigma'_0$, the flow of the vector field $u'(a, b, \tau; x) \deldel{x}$ on $(0, 1)$ takes the point $a$ to the point $b$ in time $b - a + \tau$.
  \end{itemize}
  \ynote{We use \cref{lem:u-a-b-tau} in the proof of \cref{thm:epigraph-smooth-excision}.}
\end{lemma}

\begin{proof}
  Let
  \begin{equation*}
    u'(a, b, \tau; x) := \dfrac{\frac{b-a}{2} \int_{-1}^1 \exp\Pa{-\frac{2}{1-\xi^2}} \der \xi}{\frac{b-a}{2} \int_{-1}^1 \exp\Pa{-\frac{2}{1-\xi^2}} \der \xi + \tau \exp\Pa{-\frac{(b- a)^2}{2(x - a)(b - x)}}}
  \end{equation*}
  for $x \in (a, b)$, and $u'(a, b, \tau; x) := 1$ for $x \not\in (a, b)$.
  By Barrow's formula, the flow of the vector field $u'(a, b, \tau; x) \deldel{x}$ takes $a$ to $b$ in time
  \begin{equation*}
    \int_a^b \frac{\der \xi}{u'(a, b, \tau; \xi)} = b - a + \dfrac{\tau \int_{a}^{b} \exp\Pa{-\frac{(b-a)^2}{2(\xi-a)(b-\xi)}} \der \xi}{\frac{b-a}{2} \int_{-1}^1 \exp\Pa{-\frac{2}{1-\xi^2}} \der \xi}.
  \end{equation*}
  With the change of variable $\xi = \frac{a+b}{2} + \eta \frac{b-a}{2}$, this becomes
  \begin{equation*}
    \ldots = b - a + \dfrac{\tau \, {\int_{-1}^1\exp\Pa{-\frac{2}{1-\eta^2}}}\frac{b-a}{2} \der \eta}{\frac{b-a}{2} \int_{-1}^1 \exp\Pa{-\frac{2}{1-\xi^2}} \der \xi} = b - a + \tau.
  \end{equation*}
\end{proof}

\begin{lemma} \label{lem:adjust-time-1}
  Let $\Sigma$ be a manifold.
  Let $f, a, b$ be real valued smooth functions on $\Sigma$ such that
  \begin{equation*}
    0 < f < a < b < 1.
  \end{equation*}
  Then there exists a vector field on $\Sigma \times (0, 1)$ of the form
  \begin{equation*}
    X = v(p, x) \deldel{x},
  \end{equation*}
  with $0 < v(p, x) \leq 1$, with the following properties.
  \begin{itemize}
    \item
      The forward exit time $T_X$ of $X$ satisfies
      \begin{equation*}
        T_X(p, x) \leq 1 \quad \stext{if and only if} \quad x \geq f(p).
      \end{equation*}
    \item
      The vector field $X$ coincides with $\deldel{x}$ on some neighbourhood of the set
      \begin{equation*}
        \Set{(p, x) \mmid f(p) \leq x \leq a(p) \stext{or} x \geq b(p)}.
      \end{equation*}
  \end{itemize}
  \ynote{We use \cref{lem:adjust-time-1} in the proof of \cref{thm:epigraph-smooth-excision}.}
\end{lemma}

\begin{proof}
  Let $a_1, b_1$ be real valued smooth functions on $\Sigma$ such that $a < a_1 < b_1 < b$.
  For example, we can take $a_1 = \frac23 a + \frac13 b$ and $b_1 = \frac13 a + \frac23 b$.
  Take
  \begin{equation*}
    v(p, x) := u'(a_1(p), b_1(p), f(p)); x),
  \end{equation*}
  where $u'$ is the function from \cref{lem:u-a-b-tau}.
\end{proof}

\begin{lemma} \label{lem:adjust-time-2}
  Let $\Sigma$ be a manifold.
  Let $f, h, a, b$ be real valued smooth functions on~$\Sigma$ such that
  \begin{equation*}
  0 < f < h < a < b < 1.
  \end{equation*}
  Let $X'$ be a vector field on $\Sigma \times (0, 1)$ of the form
  \begin{equation*}
    X' = v'(p, x) \deldel{x},
  \end{equation*}
  with $0 < v'(p, x) \leq 1$, with the following properties.
  \begin{itemize}
    \item
      The forward exit time $T_{X'}$ of $X'$ satisfies
      \begin{equation*}
        T_{X'}(p, x) \leq 1 \quad \stext{if and only if} \quad x \geq f(p).
      \end{equation*}
    \item
      The vector field $X'$ coincides with $\deldel{x}$ on some neighbourhood of the set $\Set{(p, x) \mmid x \geq a(p)}$.
  \end{itemize}
  Then there exists a vector field $X$ on $\Sigma \times (0, 1)$ of the form
  \begin{equation*}
    X = v(p, x) \deldel{x},
  \end{equation*}
  with $v'(p, x) \geq v(p, x) > 0$, with the following properties.
  \begin{itemize}
    \item
      The forward exit time $T_X$ of $X$ satisfies
      \begin{equation*}
        T_{X}(p, x) \leq 1 \quad \stext{if and only if} \quad x \geq h(p).
      \end{equation*}
    \item
      The vector field $X$ coincides with $X'$ on some neighbourhood of the set $\Set{(p, x) \mmid 0 < x \leq a(p)}$ and coincides with $\deldel{x}$ on some neighbourhood of the set $\Set{(p, x) \mmid x \geq b(p)}$.
  \end{itemize}
  \ynote{We use \cref{lem:adjust-time-2} in the proof of \cref{thm:epigraph-smooth-excision}.}
\end{lemma}

\begin{proof}
  By Barrow's formula, the flow of $X'$ takes each point $(p, f(p))$ to the point $(p, h(p))$ in time
  \begin{equation*}
    \tau(p) := \int_{f(p)}^{h(p)} \frac{\der x}{v'(p, x)}.
  \end{equation*}
  The function $\tau \colon \Sigma \to (0, 1)$ is smooth.
  The forward exit time $T_{X'}$ of $X'$ satisfies
  \begin{equation*}
    T_{X'}(p, h(p)) = 1 - \tau(p).
  \end{equation*}
  Let $a_1, b_1$ be real valued smooth functions on $\Sigma$ such that $a < a_1 < b_1 < b$.
  For example, we can take $a_1 = \frac23 a + \frac13 b$ and $b_1 = \frac13 a + \frac23 b$.
  Take
  \begin{equation*}
    v(p, x) :=
    \begin{cases}
      v'(p, x) & \stext{if} 0 < x \leq a_1(p), \\
      u'(a_1(p), b_1(p), \tau(p); x) & \stext{if} a_1(p) < x \leq b_1(p), \\
      1 & \stext{if} b_1(p) < x < 1,
    \end{cases}
  \end{equation*}
  where $u'$ is as in \cref{lem:u-a-b-tau}.
\end{proof}

We now use these results to excise subsets of $\Sigma \times (0, 1)$ of the form $\Set{(p, x) \mmid x \geq \lambda(p)}$ for functions $\lambda \colon \Sigma \to (0, 1]$ that are not necessarily smooth.
Recall from \cref{lem:Z-closed} that such a $\lambda$ must be lower semi-continuous.

\begin{theorem} \label{thm:epigraph-smooth-excision}
  Let $\Sigma$ be a manifold and $\lambda \colon \Sigma \to (0, 1]$ a lower semi-continuous function.
  Then there exists a vector field $X$ on $\Sigma \times (0, 1)$ of the form
  \begin{equation*}
    v(p, x) \deldel{x},
  \end{equation*}
  with $0 \leq v(p, x) \leq 1$, whose time-$1$ flow excises the subset
  \begin{equation*}
    Z := \Set{(p, x) \in \Sigma \times (0, 1) \mmid x \geq \lambda(p)}
  \end{equation*}
  from $\Sigma \times (0, 1)$.
  \ynote{We use \cref{thm:epigraph-smooth-excision}in the proof of \cref{cor:epigraph-hamiltonian-excision}.}
\end{theorem}

\begin{proof}
  By \cref{lem:time-one-excise}, it is enough to find a vector field of the required form whose forward exit time $T_X$ and backward exit time $S_X$ satisfy
  \begin{equation*}
    T_X(p, x) \leq 1 \quad \stext{if and only if} \quad x \geq \lambda(p)
  \end{equation*}
  and
  \begin{equation*}
    S_X = -\infty \quad \stext{everywhere.}
  \end{equation*}

  \cref{lem:smooth-Baire} implies that there exists a strictly increasing sequence of positive smooth functions
  \begin{equation*}
    f_n \colon \Sigma \to (0, 1) \quad \stext{for} n \in \N, \quad 0 < f_1 < f_2 < f_3 < \ldots,
  \end{equation*}
  such that
  \begin{equation*}
    \lambda = \sup\limits_{n \in \N} f_n.
  \end{equation*}

  There also exists a strictly increasing sequence of smooth functions
  \begin{equation*}
    g_n \colon \Sigma \to (0, 1), \quad g_0 < g_1 < g_2 < g_3 < \ldots
  \end{equation*}
  such that
  \begin{equation*}
    f_n < g_{n-1} \quad \stext{for all} n \in \N
  \end{equation*}
  and
  \begin{equation*}
    \sup\limits_{n \in \N} g_n \equiv 1.
  \end{equation*}
  Indeed, for $n = 0$, apply \cref{lem:g-between-one-f} with the continuous function $f_1$ to find a function $g_0 \colon \Sigma \to (0, 1)$ such that $f_1 < g_0 < 1$.
  Arguing recursively, given a smooth function $g_{n-1} \colon \Sigma \to (0, 1)$, apply \cref{lem:g-between-one-f} with the continuous function $\max \Set{g_{n-1}, f_{n+1}, 1 - \frac{1}{n}}$ to find a function $g_n \colon \Sigma \to (0, 1)$ such that $1 > g_n > \max \Set{g_{n-1}, f_{n+1}, 1 - \frac{1}{n}}$.

  We will now construct, recursively, a sequence of smooth vector fields $X_n$ on $\Sigma \times (0, 1)$ of the form
  \begin{equation*}
    X_n = v_n(p, x) \deldel{x},
  \end{equation*}
  parametrized by $n \in \N$, with the following properties.
  \begin{itemize}
    \item[(i)]
      For all $n \geq 2$ and all $p, x$,
      \begin{equation*}
        v_{n-1}(p, x) \geq v_n(p, x)  >0.
      \end{equation*}

    \item[(ii)]
      The forward exit time $T_{X_n}$ of each $X_n$ satisfies
      \begin{equation*}
        T_{X_n} (p, x) \leq 1 \quad \stext{if and only if} \quad x \geq f_n(p).
      \end{equation*}

    \item[(iii)]
      For each $n \geq 2$, the function $v_n$ coincides with the function $v_{n-1}$ on a neighbourhood of the set $\Set{(p, x) \mmid x \leq g_{n-1}(p)}$ and is equal to $1$ on a neighbourhood of the set $\Set{(p, x) \mmid x \geq g_{n}(p)}$.
  \end{itemize}
  Here is the construction.
  We obtain $X_1$ by applying \cref{lem:adjust-time-1} to the functions $f = f_1$, $a = g_0$, $b = g_1$.
  Given $X_{n-1}$, we obtain $X_n$ by applying \cref{lem:adjust-time-2} to the vector field $X' := X_{n-1}$ and to the functions $f := f_{n-1}$, $h:=f_n$, $a:=g_{n-1}$, $b:=g_n$.

  Because $\sup g_n = 1$, the union of the open sets $\Set{(x, p) \mmid x < g_n(p)}$ is all of $\Sigma \times (0, 1)$.
  By this and Property (iii), there exists a unique smooth function $v_\infty(p, x)$ on $\Sigma \times (0, 1)$ that, for each $n$, coincides with $v_n$ on $\Set{(x, p) \mmid x < g_n(p)}$.
  We define
  \begin{equation*}
    X_\infty := v_\infty(p, x) \deldel{x}.
  \end{equation*}

  We now show that, for each $(p, x) \in \Sigma \times (0, 1)$,
  \begin{equation*}
    T_{X_\infty}(p, x) > 1 \quad \stext{if and only if} \quad x < \lambda(p).
  \end{equation*}

  First, suppose that $x < \lambda(p)$.
  Because $\sup f_n = \lambda$, there exists an $n$ such that $x < f_n(p)$.
  Fix such an $n$.
  By Property (ii), $T_{X_n}(p, x) > 1$.
  Property (i) implies that $v_n(p, \cdot) \geq v_\infty(p, \cdot) > 0$ on $[x, 1)$, and so $T_{X_\infty}(p, x) \geq T_{X_n}(p, x)$.
  Putting these together, we obtain $T_{X_\infty}(p, x) > 1$.

  Now, suppose that $T_{X_\infty}(p, x) > 1$.
  Let $(p, y)$ be the image of $(p, x)$ under the time-$1$ flow of $X_\infty$.
  Because $\sup g_n = 1$ and $y < 1$, there exists an $n$ such that $y < g_n(p)$.
  Fix such an $n$.
  Because on the set $\Set{(p, x) \mmid x < g_n(p)}$ the vector field $X_\infty$ coincides with $X_n$, the time-$1$ flow of $X_n$ also takes $(p, x)$ to $(p, y)$.
  So $T_{X_n}(p, x) > 1$.
  By Property (ii), $x < f_n(p)$.
  This implies that $x < \lambda(p)$.

  To finish, take
  \begin{equation*}
    X := \chi(p, x) X_\infty
  \end{equation*}
  where $\chi(p, x) = 0$ if $0 < x \leq \frac12 f_1(p)$ and $\chi(p, x) = 1$ if $x \geq f_1(p)$.
\end{proof}

\section{Examples}
\label{sec:example}

Combining the results of Sections \ref{sec:smooth-to-hamiltonian} and~\ref{sec:smooth-epigraph-excision} yield symplectic excisions of sets with complicated topology.

\begin{proposition} \label{prop:exc-epigraph}
  Let $C$ be a closed subset of $\R^{2n-2}$.
  Then there exists a smooth real-valued function on $(\R^{2n}, \omegacan)$ whose time-$1$ Hamiltonian flow excises
  \begin{equation*}
    Z := C \times [0, \infty) \times \Set{0}
  \end{equation*}
  from $(\R^{2n}, \omegacan)$.
  \ynote{We use \cref{prop:exc-epigraph} in the proofs of \cref{exm:ray,exm:Cantor-brush}.}
\end{proposition}

\begin{proof}
  Applying \cref{lem:epigraph-smooth-forward-time} to the function $\lambda_C \equiv \frac12$, we find a vector field on $\R^{2n-2} \times (0, 1)$ of the form $v(p, x_n) \deldel{x_n}$ whose time-$1$ flow excises the subset $C \times [\frac12, 1)$ from $\R^{2n-2} \times (0, 1)$.
  Next, let $N \subset \R^{2n}$ be given by the vanishing of the $y_n$-coordinate.
  Using an order preserving diffeomorphism $(0, 1) \to \R$ that takes $[\frac12,1)$ to $[0, \infty)$, we obtain a null vector field on $N$ whose time-$1$ flow excises the subset $C \times [0, \infty) \times \Set{0}$ from $N$.
  To finish, apply \cref{thm:extension-vector-field-germ}.
\end{proof}

Now, we recover the excision of a ray:

\begin{example}[Ray] \label{exm:ray}
  There is a smooth real-valued function on $(\R^{2n}, \omegacan)$ whose time-$1$ Hamiltonian flow excises
  \begin{equation*}
    Z := \Set{0}^{2n-2} \times [0, \infty) \times \Set{0}
  \end{equation*}
  from $(\R^{2n}, \omegacan)$.
\end{example}

\begin{proof}
  This follows from \cref{prop:exc-epigraph} by taking $C = \Set{0}^{2n-2}$.
\end{proof}

More interestingly, we excise a ``Cantor brush'':

\begin{example}[Cantor brush] \label{exm:Cantor-brush}
  Let $K \subset \R$ be the cantor set.
  Then there is a smooth real-valued function on $(\R^{2n}, \omegacan)$ whose time-$1$ Hamiltonian flow excises
  \begin{equation*}
    Z := K \times \Set{0}^{2n-3} \times [0, \infty) \times \Set{0}
  \end{equation*}
  from $(\R^{2n}, \omegacan)$.
  See \cref{fig:Cantor-brush}.
\end{example}

\begin{proof}
  This follows from \cref{prop:exc-epigraph} by taking $C = K \times \Set{0}^{2n-3}$.
\end{proof}

\begin{figure}[htb]
  \centering
  \inputfigure{excision-c}
  \caption{Removing the Cantor brush.}
  \label{fig:Cantor-brush}
\end{figure}

\begin{proposition} \label{cor:epigraph-hamiltonian-excision}
  Let $\lambda \colon \R^{2n-2} \to (-\infty, \infty]$ be a lower semi-continuous function.
  Then there is a smooth real-valued function on $(\R^{2n}, \omegacan)$ whose time-$1$ Hamiltonian flow excises
  \begin{equation*}
    Z := \Set{(p; x_n, y_n) \in \R^{2n-2} \times \R^2 \mmid x_n \geq \lambda(p) \stext{and} y_n = 0}
  \end{equation*}
  from $(\R^{2n}, \omegacan)$.
  \ynote{We use \cref{cor:epigraph-hamiltonian-excision} in the proofs of \cref{exm:box-w-tail,exm:angle}.}
\end{proposition}

\begin{proof}
  Combining \cref{thm:epigraph-smooth-excision} with an order-preserving diffeomorphism $(0, 1) \to \R$, we obtain a vector field on $\R^{2n-1}$ that is a multiple of $\deldel{x_n}$, whose time-$1$ flow excises
  \begin{equation*}
    Z_0 := \Set{(p, x_n) \in \R^{2n-2} \times \R \mmid x_n \geq \lambda(p)}
  \end{equation*}
  from $\R^{2n-1}$.
  We then identify $\R^{2n-1}$ with the subset $\R^{2n-1} \times \Set{0}$ of $\R^{2n}$, and apply \cref{thm:extension-vector-field-germ}.
\end{proof}

\begin{example}[Box with a tail] \label{exm:box-w-tail}
  Let $n \geq 2$.
  There is a smooth real-valued function on $(\R^{2n}, \omegacan)$ whose time-$1$ Hamiltonian flow excises the ``box with a tail''
  \begin{equation*}
    Z := \Pa{[-1, 1]^{2n-2} \times [0, \infty) \cup \Set{0}^{2n-2} \times [-1, \infty)} \times \Set{0}
  \end{equation*}
  from $(\R^{2n}, \omegacan)$. See \cref{fig:box-tail}.
  \ynote{We use \cref{exm:box-w-tail} in the proof of \cref{exm:narrow nbhd}.}
\end{example}

\begin{proof}
  This follows from \cref{cor:epigraph-hamiltonian-excision}, with the lower semi-continuous function $\lambda \colon \R^{2n-2} \to (-\infty, \infty]$ that takes $\Set{0}^{2n-2}$ to~$-1$, that takes $[-1, 1]^{2n-2} \setminus \Set{0}^{2n-2}$ to $0$, and that takes $\R^{2n-2} \setminus [-1, 1]^{2n-2}$ to $\infty$.
\end{proof}

\begin{figure}[htb]
  \centering
  \inputfigure{excision-b}
  \caption{Removing a box with a tail.}
  \label{fig:box-tail}
\end{figure}

\begin{example}[Cone] \label{exm:angle}
  Let $n \geq 2$.
  There is a smooth real-valued function on $(\R^{2n}, \omegacan)$ whose time-$1$ Hamiltonian flow excises the ``cone''
  \begin{equation*}
    Z := \Set{z \in \R^{2n} \mmid x_n \geq \sqrt{x_1^2 + y_1^2 + \dotsb + x_{n-1}^2 + y_{n-1}^2}, y_n = 0}.
  \end{equation*}
  from $(\R^{2n}, \omegacan)$.
  See \cref{fig:angle}.
\end{example}

\begin{proof}
  This follows from \cref{cor:epigraph-hamiltonian-excision}, with the lower semi-continuous function
  \begin{equation*} \begin{split}
    \lambda \colon \R^{2n-2} &\to (-\infty, \infty] \\
    (x_1, y_1, \dotsc, x_{n-1}, y_{n-1}) &\mapsto \sqrt{x_1^2 + y_1^2 + \dotsb + x_{n-1}^2 + y_{n-1}^2}.
  \end{split} \end{equation*}
\end{proof}

\begin{figure}[htb]
  \centering
  \inputfigure{excision-a}
  \caption{Removing a cone.}
  \label{fig:angle}
\end{figure}

\begin{example}[Narrowing neighbourhood] \label{exm:narrow nbhd}
  Let $n \geq 2$, let $Z \subset \R^{2n}$ be the ``box with a tail'' as in \cref{exm:box-w-tail}, and let $M \subset \R^{2n}$ be the open subset
  \begin{equation*}
    M := \Set{z \in \R^{2n} \mmid y_n \leq 0 \stext{or} \operatorname{dist}(z, Z) < \frac{1}{y_n}}.
  \end{equation*}
  Then there is a smooth real-valued function on $M$ whose time-$1$ Hamiltonian flow excises $Z$ from $M$.
\end{example}

\begin{proof}
  Combine \cref{exm:box-w-tail} with \cref{thm:locality-excision}.
\end{proof}

\begin{remark} \label{rem:ray-two-horns-auto}
  In \cref{cor:ray-two-horns} we saw how to excise a ``ray with two horns'' $Z$ from a symplectic manifold $M$ by composing of three time-$1$ maps of Hamiltonian flows.
  Interestingly, $Z$ itself cannot be excised by a time-$1$ flow of a vector field $X$.
  Indeed, assume otherwise.
  Let $x$ be the point on $Z$ whose neighbourhood in $Z$ is a union of three intervals emanating from the point.
  Let $0 < \eps < T_X(x)$. The time-$\eps$ flow of $X$ is a diffeomorphism $U \to V$ beween open subsets $U$, $V$ of $M$, which takes $x$ to a point $y \neq x$.
  By \cref{lem:time-one-excise}, it takes $U \cap Z$ into $V \cap Z$.
  But $y$ does not have any neighbourhood in $Z$ that is homeomorphic to a neighbourhood of $x$ in $Z$.
\eor
\end{remark}

\printbibliography

\authaddresses
\end{document}